\documentclass[11pt,a4paper]{amsart}
\usepackage{amsmath,amssymb, amsbsy}
\usepackage{subfigure}
\usepackage{graphpap,latexsym,epsf}
\usepackage{color,psfrag}
\usepackage[dvips]{graphicx}
\usepackage{bm}

\usepackage[english]{babel}
\usepackage{times}
\usepackage{graphicx}
\usepackage{amscd}
\usepackage{amsmath}
\usepackage{amsfonts}
\usepackage{amssymb}
\usepackage{amsthm}
\usepackage{latexsym}

\numberwithin{equation}{section}
\newtheorem{theorem}{Theorem}[section]
\newtheorem{proposition}[theorem]{Proposition}
\newtheorem{lemma}[theorem]{Lemma}

\newtheorem{definition}[theorem]{Definition}

\newtheorem{remark}{Remark}[section]

\def\neweq#1{\begin{equation}\label{#1}}
\def\endeq{\end{equation}}
\def\eq#1{(\ref{#1})}

\newcommand{\R}{\mathbb{R}}
\newcommand{\N}{\mathbb{N}}

\newcommand{\eps}{\varepsilon}

\newcommand{\well}{\widetilde\ell}

\newcommand{\SI}{{\bm \sigma}}
\newcommand{\e}{{\bf e}}
\newcommand{\EC}{\overset{E}{\longrightarrow}}
\newcommand{\EEC}{\overset{EE}{\longrightarrow}}
\newcommand{\CC}{\overset{C}{\longrightarrow}}

\begin{document}

\title{An orthotropic plate model for decks of suspension bridges}

\author{Alberto Ferrero}

\address{\hbox{\parbox{5.7in}{\medskip\noindent{Alberto Ferrero, \\
Universit\`a del Piemonte Orientale, \\
        Dipartimento di Scienze e Innovazione Tecnologica, \\
        Viale Teresa Michel 11, 15121 Alessandria, Italy. \\[5pt]
        \em{E-mail address: }{\tt alberto.ferrero@uniupo.it}}}}}

\date{\today}

\maketitle

\begin{abstract} The main purpose of the present paper is to compare two different kinds of approaches in modeling the deck of a suspension bridge: in the first approach we look at the deck as a rectangular plate and in the second one we look at the deck as a beam for vertical deflections and as a rod for torsional deformations.
Throughout this paper we will refer to the model corresponding to the second approach as the \textit{beam-rod model}.
In our discussion, we observe that the beam-rod model has more \textit{degrees of freedom} if compared with the isotropic plate model. For this reason the beam-rod model is supposed to be more appropiate to describe the behavior of the deck of a real suspension bridge. A possible strategy to make the plate model more efficient could be to relax the isotropy condition with a more general condition of orthotropy, which is expected to increase the degrees of freedom in view of the larger number of elastic parameters. In this new setting, a comparison between the two approaches becomes now possible.

Basic results are proved for the suggested problem, from existence and uniqueness of solutions to spectral properties.
We suggest realistic values for the elastic parameters thus obtaining with both approaches similar responses in the static and dynamic behavior of the deck. This can be considered as a preliminary article since many work has still to be done with the perspective of formulating models for a complete suspension bridge which take into account not only the deck but also the action on it of cables and hangers. With this perspective, a section is devoted to possible future developments.

\end{abstract}

\vspace{11pt}


\vspace{6pt}

\section{Introduction} \label{s:introduction}

The main purpose of this paper is to develop an alternative approach in modeling decks of suspension bridges and to compare it with the existing ones. In recent years, several authors dealt with different models describing static and dynamic behavior of suspension bridges; we quote \cite{Gazzola-Libro} for a survey on the scientific literature on this topic.

\begin{figure}[th]
\begin{center}
 {\includegraphics[scale=0.45]{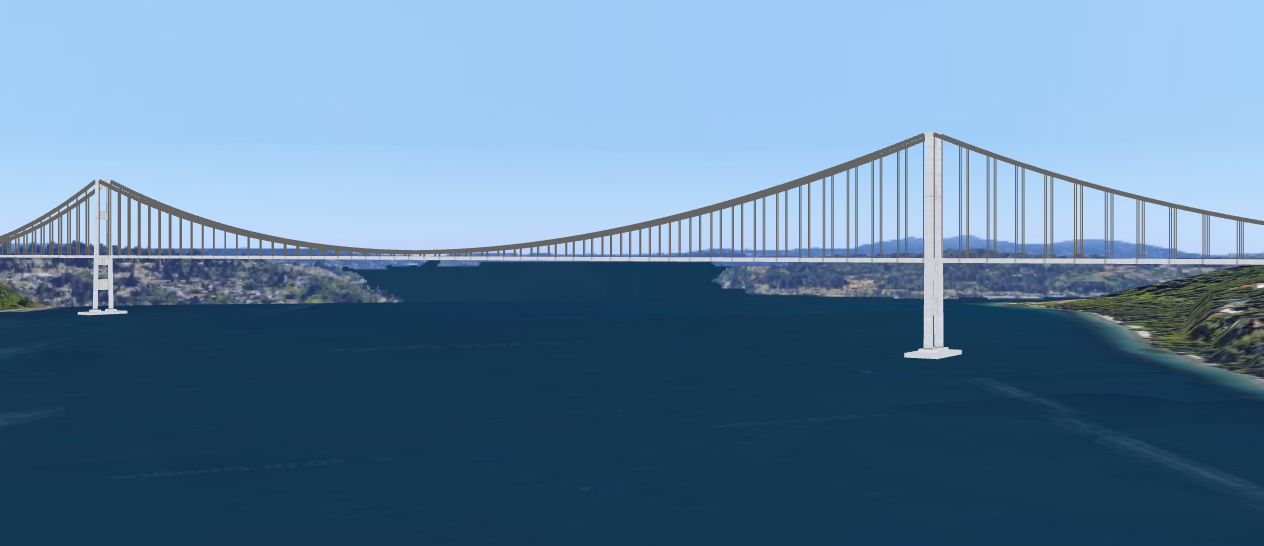}}
\caption{A typical suspension bridge with two towers and three spans.} \label{tacoma}
\end{center}
\end{figure}

A typical suspension bridge, like the one in Figure \ref{tacoma}, is made up of several parts: two towers, one central span, two sides spans, two cables for each span, the deck and the hangers connecting the cables to the deck. At the present state of art, most of the papers that can be found in the literature, are essentially focused on the behavior of a single span, typically the central one which, for obvious reasons, is the one which can show a higher level of instability due to its length. We have in mind what happened for example to the Tacoma Narrows Bridge when in 1940, its central span collapsed in the underlying channel, see \cite{Tacoma}. The Federal Report \cite{ammann} by O. H. Ammann, T. von K\'arm\'an, G. B. Woodruff, considers ``the crucial event in the collapse to be the sudden change from a vertical to a torsional mode of oscillation''.
In \cite{ammann} one can find a first one-dimensional model describing vertical oscillations but, as the above statement by the authors of \cite{ammann} suggests, models describing also torsional deformations assume a great relevance.

More recently, the authors of \cite{LazMcK,McKWa,McKWa-2} suggested one-dimensional models consisting of a beam suspended by hangers: the action of the hangers is described by a restoring force which is proportional to the stretching when the hangers are stretched and which vanishes when the hangers are slackened.

In \cite{McK,McKT} the authors suggested models in which both vertical and torsional oscillations were considered simultaneously: the cross section of the deck was described by a rod free to rotate about its center and to translate vertically. They were able to describe the sudden transition from standard vertical oscillations to destructive torsional oscillations.

Later, the authors of \cite{AG-Many-Rods} reconsidered the model by \cite{McK,McKT} describing the entire bridge by mean of a finite number of parallel rods linked to the two nearest neighbors rods with attractive linear forces representing resistance
to longitudinal and torsional stretching. The sudden appearance of torsional oscillations was highlighted also in the multiple rods approach.

The aim the present paper is to compare two other kinds of approaches in the description of the deck of a bridge: in the first one, a single equation is considered by interpreting the deck as an isotropic plate; in the second one, a system of two equations is introduced with the purpose of describing vertical deflections and torsional deformations respectively.
We now proceed with a detailed description of these two approaches.

Let us start with the system model, in which a beam is used to describe vertical oscillations of the deck and a rod is used to describe its torsional oscillations. The resulting elastic energy can be split into two components given by the bending energy of the beam and the torsional energy of the rod, namely
\begin{equation} \label{eq:energy-v-t}
\mathbb E_{BT}(\psi,\theta)=\frac{EI}2 \int_0^L |\psi''(x)|^2 \, dx+\frac{\mu
K}2\int_0^L |\theta'(x)|^2 \, dx
\end{equation}
where $\psi=\psi(x)$ denotes the vertical displacement of the middle line of the deck while $\theta=\theta(x)$ denotes the angle of torsion, as shown in Figure \ref{f:vertical-torsional}.
The authors of \cite{BeGa} called this model {\it fish bone} for the simultaneous presence of the longitudinal midline of the road and the transversal cross sections, as one can see from \cite[Figure 1]{BeGa}.

\begin{figure}[th]
\begin{center}
 {\includegraphics[scale=0.30]{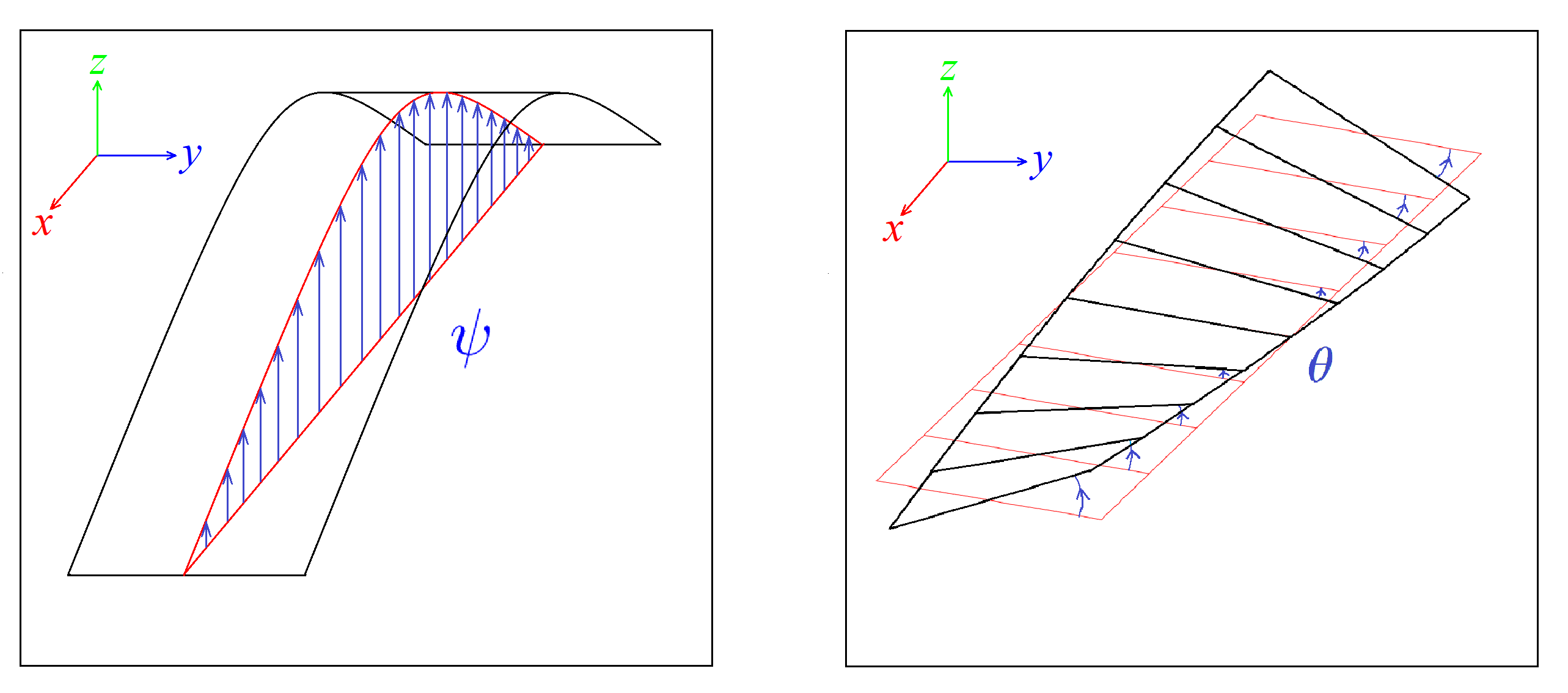}}
\caption{The vertical displacement $u$ and the torsion angle $\theta$ of the deck.} \label{f:vertical-torsional}
\end{center}
\end{figure}

The coefficients appearing in \eqref{eq:energy-v-t} come from the theory of linear elasticity applied to real structures like beams and rods: $\mu$ is one of the two Lam\'e constants and precisely the shear modulus which can be also expressed in terms of the Young modulus $E$ and the
Poisson ratio $\nu$, namely $\mu=\frac E{2(1+\nu)}$ as one can see from \cite[(5.9)]{LaLi}; here the deck of the bridge at the rest position is described by the rectangle parallelepiped $(0,L)\times (-\ell,\ell)\times \left(-\frac d2, \frac d2\right)$ and the coefficients $I$ and $K$ are given respectively by
\begin{equation} \label{eq:I-K}
  I=\int_{\mathcal S} z^2 dydz=\frac{d^3 \ell}6 \qquad \text{and} \qquad K=4\int_{\mathcal S} |\nabla \chi(y,z)|^2 \, dy dz
\end{equation}
where $\mathcal S=(-\ell,\ell)\times \left(-\frac d 2,\frac d 2\right)$ and $\chi$ is the solution of
\begin{equation*}
\begin{cases}
-\Delta \chi=1 & \quad \text{in } \mathcal S \, ,\\
\chi=0 & \quad \text{on } \partial\mathcal S \, ,
\end{cases}
\end{equation*}
see \cite[Paragraph 16]{LaLi} for more details.

More precisely, looking at Figure \ref{f:vertical-torsional}, we see that the deck of the bridge has length $L$, width $2\ell$ and thickness $d$, so that the coefficient $I$ represents the moment of inertia of the cross section of the deck with respect to the symmetry axis parallel to the $y$ axis and $K$ is the torsional constant of the deck which multiplied by $\mu$ gives its torsional rigidity.

Applying to the deck a vertical load per unit of length $F$ and a moment of forces per unit of length $M$, by \eqref{eq:energy-v-t}, the Euler-Lagrange equations become
\begin{equation} \label{eq:system-u-theta}
  \begin{cases}
    EI \psi''''=F  & \qquad \text{in } (0,L) \, ,\\[6pt]
    -\mu K \theta''=M & \qquad \text{in } (0,L) \, .
  \end{cases}
\end{equation}
For more details on this model we quote \cite{AG15, AG17, BeGa, LaLi} and the references therein.
In the sequel we will refer to \eqref{eq:system-u-theta} as the \textit{beam-rod model}.

As mentioned above, in the other approach the deck is interpreted as a plate of length $L$, width $2\ell$ and thickness $d$. Denoting as above by $E$ and $\nu$ the Young modulus and Poisson ratio respectively, and by $u=u(x,y)$ the vertical displacement of the middle surface of the plate, the resulting elastic bending energy becomes
\begin{equation}\label{eq:plate-isotropic}
  \mathbb E_B(u)=\frac{Ed^3}{24(1-\nu^2)} \int_{(0,L)\times (-\ell,\ell)}\left(\nu |\Delta u|^2+(1-\nu)|D^2 u|^2\right)dxdy
\end{equation}
where $D^2 u$ denotes the Hessian matrix of $u$ and $|D^2 u|=\left(u_{xx}^2+2u_{xy}^2+u_{yy}^2\right)^{1/2}$ is its Euclidean norm, obtained by interpreting the matrix as a vector of four components. For more details on the plate model see \cite{BeBuoGa, BeFeGa, FeGa, LaLi}.
This is known as the Kirchhoff-Love model for a plate made of an isotropic material, see \cite{Kirchhoff, Love}.

Denoting by $f$ a vertical load per unit of surface, the corresponding Euler-Lagrange equation becomes
\begin{equation} \label{eq:plate-iso-EuLa}
\frac{Ed^3}{12(1-\nu^2)} \, \Delta^2 u=f \qquad \text{in } \Omega=(0,L)\times (-\ell,\ell) \, .
\end{equation}
The constant $\frac{Ed^3}{12(1-\nu^2)}$ in front of the biharmonic operator $\Delta^2$ represents the rigidity of the plate.

In order to have a more detailed picture on the recent literature on rectangular plates and applications in models for decks of bridges, we quote \cite{AnGa, BeBuGaZu, BeFaFeGa, BoGaMo, ChGaGa} and the references therein.

From a comparison between \eqref{eq:system-u-theta} and \eqref{eq:plate-iso-EuLa}, it appears clear that the beam-rod model allows more degrees of freedom since the bending rigidity and the torsional rigidity of the deck may be chosen independently one from the other; conversely, in the model of the plate the choice of the Young modulus and the Poisson ratio affects the bending and torsional rigidities simultaneously.

As a result, we have that two models describing the same structure with the same sizes could produce quite different responses.
However, it is well known that concerning vertical displacements, the behavior of a plate of small width is comparable with the one of the beam with the same sizes; for a rigorous result one can look at \cite[Theorem 3.3]{FeGa}.
On the other hand, concerning the torsional behavior, we believe that, even for small width, the plate equation and the second equation in \eqref{eq:system-u-theta} could produce different responses.

In order to overcome this inconvenient, once the bending rigidity of the plate is determined, one can try to modify in a forced way the width of the plate thus giving rise to the model of an ideal plate with a non realistic width.
This is not the path we want to follow.
Indeed, we believe that a more realistic way to proceed is to formulate a new model of plate exhibiting an anisotropic behavior. Considering a plate made of an anisotropic material is expected to produce more parameters in the model as a consequence of the increased number of elastic coefficients of the material.
Such an increased degree of freedom in the choice of the parameters allows to find a suitable combination of them producing a more realistic model of plate with a behavior closer to the one of the beam-rod model.

Our idea is to construct a model of plate made of an orthotropic material (see Section \ref{s:orthotropic} for the rigorous definition of orthotropy) exhibiting more rigidity in the $x$ direction than in the $y$ and $z$ directions.
The resulting bending energy of the orthotropic plate can be found in \eqref{eq:E-B}. The corresponding Euler-Lagrange equation for the plate subject to a vertical load per unit of surface then becomes
\begin{equation*}
  \frac{d^3\mathcal K}{12}\left(\Delta^2 u+\kappa \frac{\partial^4 u}{\partial x^4}\right)=f
\end{equation*}
where $\mathcal K=\frac{E_2}{1-\nu_{12}\nu_{21}}$ and $\kappa=\frac{E_1-E_2}{E_2}$ being $E_1>E_2=E_3$ the Young moduli in the $x$, $y$, $z$ directions respectively, and $\nu_{12}$ and $\nu_{21}$ the Poisson ratios relative to the $x$ and $y$ directions, see Section \ref{ss:Ortho-Mat} for the correct definitions. In this way we may write the equation in the more familiar form
\begin{equation} \label{eq:familiar}
 \frac{E_1 d^3}{12(1-\nu_{12} \nu_{21})} \, \frac{\partial^4 u}{\partial x^4}
 +\frac{E_2 d^3}{6(1-\nu_{12} \nu_{21})} \, \frac{\partial^4 u}{\partial x^2\partial y^2}
 +\frac{E_2 d^3}{12(1-\nu_{12} \nu_{21})} \, \frac{\partial^4 u}{\partial y^4}=f \, ,
\end{equation}
see also Remark \ref{r:1}.

In Section \ref{s:orthotropic} we make a survey of some well known notions about elastic anisotropic materials like the anisotropic Hooke's law and the related {\it stiffness matrix}. We explain how the stiffness matrix looks like in the case of an orthotropic material, we show how the elastic energy per unit of volume can be expressed in terms of the components of the strain tensor and we describe the meaning of the elastic coefficients involved in the orthotropic Hooke's law, still named Young moduli (one for each coordinate axis) and Poisson ratios (one for each combination of two coordinate axes); other coefficients completely independent of the previous ones are the so called {\it shear moduli}.
The last part of Section \ref{s:orthotropic} is devoted to a description of an orthotropic material with a two-dimensional symmetry (with respect to planes parallel to the $yz$ plane) and a one-dimensional reinforcement in the orthogonal direction (the $x$ axis).

Having in mind the explanations given in Section \ref{s:orthotropic} about orthotropic materials, in Section \ref{s:orthotropic-plate} we construct a model of orthotropic plate with a one-dimensional reinforcement, giving rise to the equation \eqref{eq:plate-iso-EuLa}. Since our purpose is to describe the deck of a bridge, the rectangular plate is supposed to be hinged at the two shorter edges which corresponds to have a solution of \eqref{eq:plate-iso-EuLa} subject to homogeneous Navier boundary conditions on these two edges and free boundary conditions on the other two edges.
We prove in Theorem \ref{t:Lax-Milgram} existence and uniqueness of weak solutions for the corresponding boundary value problem; in the same statement we also prove a regularity result based on the classical elliptic estimates by \cite{adn}.

In the same section we deal with another crucial part, consisting in the study of the spectral properties of the operator
$\Delta^2+\kappa \partial^4_x$ coupled with the mixed boundary conditions (Navier boundary conditions on the shorter edges and free boundary conditions on the other two edges). An implicit representation of the eigenvalues and of the corresponding eigenfunctions is provided in Theorem \ref{t:eigenvalues}.

In Section \ref{s:small-width}, we study the behavior of an orthotropic rectangular plate in which its width is small if compared with its length, as it happens in a real bridge. In Theorem \ref{t:conv-poisson}, we prove that the solution $u_\ell$ of the plate equation converges to the solution of the beam equation when $\ell\to 0$.
This result extends to the orthotropic case an analogous result proved in \cite{FeGa} in the case of the isotropic plate.

The second main result of Section \ref{s:small-width} is Theorem \ref{t:spectral-convergence}, where it is proved convergence of the eigenvalues of $\Delta^2+\kappa \partial^4_x$ to the eigenvalues of the beam eigenvalue problem \eqref{eq:beam-eigenvalue}.

Section \ref{s:static} is devoted to the comparison between the orthotropic plate model and the beam-rod model. Our purpose is to determine the values of the parameters in the orthotropic plate equation \eqref{eq:familiar} which make the behavior of the plate completely comparable with the one of the beam-rod system.

We recall that $E$ is the Young modulus used in the beam model and that $E_1$ and $E_2=E_3$ denote the Young moduli of the plate with respect to the $x$, $y$ and $z$ directions respectively. We simply denote by $\nu$ the Poisson ratio $\nu_{12}$ and
we assign to it the value $0.2$ as one can see from \eqref{eq:coeff-E1-E2-nu}. By \eqref{eq:id0}, we infer that the other Poisson ratio $\nu_{21}$ can be expressed in terms of $\nu_{12}$, $E_1$ and $E_2$.
In this way, looking at \eqref{eq:familiar}, we see that the behavior of the plate is completely determined once we find appropriate values for $E_1$ and $E_2$.

In order to fulfill this purpose, we first force the plate to have the same bending rigidity of the beam, thus obtaining $E_1=E$.
Then, testing the torsional rigidity of the plate and comparing it with the torsional rigidity $\mathcal R_T=\mu \, K$ in \eqref{eq:system-u-theta}, we achieve our goal as one can see by \eqref{eq:val-E2} where $E_2$ is expressed in terms of $E$, $I$, $\mathcal R_T$ and $\nu$.

The numerical values of the bending rigidity $EI$ and the torsional rigidity $\mathcal R_T$ are obtained having in mind the structural parameters of the Tacoma Narrows Bridge collapsed in 1940. This numerical values for the Tacoma Narrows Bridge are deduced from the papers \cite{AG17,Plaut,Tullini}.

As the reader can realize looking at \eqref{eq:coeff-E1-E2-nu}, the procedure described above produces a Young modulus $E_1$ which is of two orders of magnitude larger than the Young modulus $E_2$, thus showing that the plate exhibits a strongly anisotropic behavior and confirming that the more popular isotropic plate model is not completely suitable to describe torsional oscillations of the deck of a bridge.

Once we have obtained the numerical values of all the parameters is then possible to compare the responses of the two static models testing them with suitable loads generating either vertical or torsional displacements.
We see that the difference between the responses corresponding to the two different approaches, is negligible in the case of vertical displacements and more relevant but relatively small in the case of torsion.

In Section \ref{s:frequence}, we also analyze and compare the frequencies of free vibrations of the deck obtained with the two different approaches and as a result we confirm what we have already observed in Section \ref{s:static} about static behavior: a negligible difference between the frequencies of vertical vibration and a more relevant but relatively small difference concerning torsional vibrations.

We devoted Section \ref{s:develop} to possible future developments in the direction of formulating an evolution model describing the dynamic behavior of the entire central span of a suspension bridge comprehensive of the deck, of the two cables, of the hangers and of their reciprocal dynamic behavior.
After that, it will probably be possible to have a more clear picture of how much the two approaches treated in this article are comparable.
The remaining sections are devoted to the proofs of the main results.


\section{Orthotropic materials} \label{s:orthotropic}

\subsection{Basic notions on the theory of anisotropic materials}
We recall some notations from the theory of linear elasticity. We denote by $\SI=(\sigma_{ij})$ the stress tensor and by ${\bf e}=(e_{ij})$  the strain tensor with $i,j\in \{1,2,3\}$.
More precisely, by strain tensor ${\bf e}$ we actually mean the \textit{linearized strain tensor}, i.e.
\begin{equation} \label{eq:strain-tensor}
e_{ij}=\frac 12 \left(\frac{\partial u_i}{\partial
x_j}+\frac{\partial u_j}{\partial x_i}\right) \, , \qquad i,j\in
\{1,\dots,3\}
\end{equation}
where ${\bf u}=(u_1,u_2,u_3)$ is the displacement vector field of the elastic body.

We recall that both ${\bm \sigma}$ and ${\bf e}$ are symmetric tensors, i.e. $\sigma_{ij}=\sigma_{ji}$ and $e_{ij}=e_{ji}$ for any $i,j\in \{1,2,3\}$.
See the book by Landau \& Lifshitz \cite{LaLi} for more details on these basic notions.

In the theory of linear elasticity, the tension of a material as a consequence of a deformation is proportional to the deformation itself and conversely the deformation of a material is proportional to the forces acting on it. This notion can be resumed by the generalized Hooke's law that states the existence of $81$ coefficients $(C_{ijkl})$ with $i,j,k,l\in \{1,2,3\}$ such that
\begin{equation} \label{eq:Hook-0}
  \sigma_{ij}=\sum_{k,l=1}^{3} C_{ijkl}\,  e_{kl}   \qquad \text{for any } i,j\in \{1,2,3\} \, .
\end{equation}
The elastic energy per unit of volume $\mathcal E$, as a function of the components of the strain tensor, can be implicitly characterized by
\begin{equation} \label{eq:implcit}
\frac{\partial \mathcal E}{\partial e_{ij}}=\sigma_{ij} \, , \ \ \text{for any } i,j\in \{1,2,,3\} \, , \quad   \mathcal E=0
  \ \ \ \text{when the material is undeformed} \, ,
\end{equation}
see \cite[Chapter 1, Paragraph 2]{LaLi} for more explanations on this question.

Then, by \eqref{eq:Hook-0}, we have that
\begin{equation} \label{eq:second-der}
  \frac{\partial^2 \mathcal E}{\partial e_{ij}\partial e_{kl}}=C_{ijkl} \qquad \text{for any } i,j,k,l\in \{1,2,3\}
\end{equation}
thus showing the following symmetry property
\begin{equation} \label{eq:symmetry-1}
  C_{ijkl}=C_{klij}  \qquad \text{for any } i,j,k,l\in \{1,2,3\}  \, .
\end{equation}
Combining \eqref{eq:implcit}, \eqref{eq:second-der} and \eqref{eq:symmetry-1} we obtain the explicit representation of the elastic energy per unit of volume:
\begin{equation} \label{eq:explicit}
 \mathcal E=\frac 12 \sum_{i,j,k,l=1}^{3} C_{ijkl} \, e_{ij} e_{kl} \, .
\end{equation}

Since the tensors $\SI$ and $\e$ are symmetric, actually each of them is completely determined by only six components
that we represent here as vector columns
\begin{align*}
& \widetilde{\bm\sigma}=
\begin{pmatrix}
\sigma_{11} & \sigma_{22} & \sigma_{33} & \sigma_{12} & \sigma_{13} & \sigma_{23}
\end{pmatrix}^T \, ,
\qquad
\widetilde{\bm e}=
\begin{pmatrix}
e_{11} & e_{22} & e_{33} & e_{12} & e_{13} & e_{23}
\end{pmatrix}^T \, .
\end{align*}

The related Hooke's law admits the following matrix representation
\begin{equation} \label{eq:def-C}
  \widetilde{\bm\sigma}=C \widetilde{\bm e}
\end{equation}
where $C$ is a $6\times 6$ matrix called \textit{stiffness matrix}.

\subsection{Orthotropic materials} \label{ss:Ortho-Mat}
A material is said to be orthotropic if the corresponding stiffness matrix remains invariant under reflections with respect to three mutually orthogonal planes. In other words there exists three
orthogonal symmetry planes for which the stiffness matrix remains invariant under the corresponding reflections.

In this section we collect the basic notions and formulas that will be needed in the subsequent part and for more details and justifications we refer to \cite{Orthotropic}.

It is well known that for an orthotropic material, the Hooke's law \eqref{eq:def-C} becomes
{\small
\begin{equation} \label{eq:Hook-1}
\begin{pmatrix}
\sigma_{11} \\
\sigma_{22} \\
\sigma_{33} \\
\sigma_{12} \\
\sigma_{13} \\
\sigma_{23} \\
\end{pmatrix}
=
\begin{pmatrix}
C_{1111} & C_{1122} & C_{1133} & 0 & 0 & 0 \\
C_{1122} & C_{2222} & C_{2233} & 0 & 0 & 0 \\
C_{1133} & C_{2233} & C_{3333} & 0 & 0 & 0 \\
0 & 0 & 0 & C_{1212} & 0 & 0 \\
0 & 0 & 0 & 0 & C_{1313} & 0 \\
0 & 0 & 0 & 0 & 0 & C_{2323}
\end{pmatrix}
\begin{pmatrix}
e_{11} \\
e_{22} \\
e_{33} \\
e_{12} \\
e_{13} \\
e_{23} \\
\end{pmatrix} \, ,
\end{equation}
}
see \cite[Section 2.2]{Orthotropic}.

Let us denote by ${\bf B_1}$ the $3\times 3$ North-West block of the stiffness matrix in \eqref{eq:Hook-1} and let us put
$$
  \e_{{\rm diag}}=\begin{pmatrix} e_{11} & e_{22} & e_{33} \end{pmatrix}^T \, .
$$
With this notation, it may be shown that \eqref{eq:explicit}, may be written in the form

\begin{align}\label{eq:elastic-energy}
  \mathcal E& =\frac 12 \,
  \e_{{\rm diag}}^T  \,
  {\bf B_1} \, \e_{{\rm diag}}
  +C_{1212}\, e_{12}^2
  +C_{1313}\, e_{13}^2
  +C_{2323}\, e_{23}^2 \, ,
\end{align}
see \cite[Section 2.3]{Orthotropic} for more details.

Consider now the inverse of identity \eqref{eq:Hook-1}, in which the components of
the strain tensor are expressed in terms of the components of the stress tensor,
\begin{equation} \label{eq:Hook-2}
\begin{pmatrix}
e_{11} \\
e_{22} \\
e_{33} \\
e_{12} \\
e_{13} \\
e_{23} \\
\end{pmatrix}
=
\begin{pmatrix}
\frac 1{E_{1}} & -\frac{\nu_{21}}{E_{2}} & -\frac{\nu_{31}}{E_{3}} & 0 & 0 & 0 \\
-\frac{\nu_{12}}{E_{1}} & \frac 1{E_{2}} & -\frac{\nu_{32}}{E_{3}} & 0 & 0 & 0 \\
-\frac{\nu_{13}}{E_{1}} & -\frac{\nu_{23}}{E_{2}} & \frac 1{E_{3}} & 0 & 0 & 0 \\
0 & 0 & 0 & \frac 1{2\mu_{12}} & 0 & 0 \\
0 & 0 & 0 & 0 & \frac 1{2\mu_{13}} & 0 \\
0 & 0 & 0 & 0 & 0 & \frac 1{2\mu_{23}}
\end{pmatrix}
\begin{pmatrix}
\sigma_{11} \\
\sigma_{22} \\
\sigma_{33} \\
\sigma_{12} \\
\sigma_{13} \\
\sigma_{23} \\
\end{pmatrix}
\end{equation}
where the constants $E_{1}, E_{2}, E_{3}$ are known as Young
moduli with respect to the three directions and the constants
$\nu_{ij}$, $i,j\in \{1,2,3\}$, $i\neq j$ are known as Poisson
ratios.

Let us denote by $S$ the $6\times 6$ matrix appearing in \eqref{eq:Hook-2} which clearly coincides with $C^{-1}$ as one can see by comparing \eqref{eq:Hook-1} with \eqref{eq:Hook-2}.
The symmetry properties of the stiffening matrix $C$ is inherited by $S$ so that
\begin{equation} \label{eq:id0}
\frac{\nu_{21}}{E_{2}}=\frac{\nu_{12}}{E_{1}} \, , \qquad
\frac{\nu_{31}}{E_{3}}=\frac{\nu_{13}}{E_{1}} \, , \qquad
\frac{\nu_{32}}{E_{3}}=\frac{\nu_{23}}{E_{2}} \, .
\end{equation}
We observe that by \eqref{eq:Hook-2}, in the
case of a one-dimensional tension state parallel to the $x_1$-axis, i.e.
the only component of $\SI$ different from zero is $\sigma_{11}$,
we have that
\begin{equation*}
\nu_{12}=-\tfrac{e_{22}}{e_{11}} \quad \text{and} \quad
\nu_{13}=-\tfrac{e_{33}}{e_{11}} \, .
\end{equation*}
Similarly, choosing first a one-dimensional tension state parallel
to the $x_2$-axis and then a one-dimensional tension state parallel to the $x_3$-axis, we infer
\begin{equation*}
\nu_{21}=-\tfrac{e_{11}}{e_{22}} \, , \qquad \nu_{23}=-\tfrac{e_{33}}{e_{22}} \, , \qquad
\nu_{31}=-\tfrac{e_{11}}{e_{33}} \, , \qquad \nu_{32}=-\tfrac{e_{22}}{e_{33}} \, .
\end{equation*}
In other words, if we consider the Poisson ratio $\nu_{ij}$, the first index $i$ represents the direction of the
one-dimensional stress and the second index $j$ represents the direction of the transversal deformation.
This explanation clarifies the meaning of the Poisson ratios and the notation used in \eqref{eq:Hook-2}.

Finally $\mu_{12}, \mu_{13}, \mu_{23}$ are known as shear moduli
or moduli of rigidity and their indices coincide with the corresponding components of the stress and strain tensors obviously involved by these coefficients:
\begin{equation*}
  e_{12}=(2\mu_{12})^{-1} \, \sigma_{12} \, , \qquad e_{13}=(2\mu_{13})^{-1} \, \sigma_{13} \, ,\qquad e_{23}=(2\mu_{23})^{-1} \, \sigma_{23} \, .
\end{equation*}

In the sequel we need to represent the components of $C$ in terms of the coefficients $E_i$, $\nu_{ij}$, $\mu_{ij}$; this may be simply done with a procedure of inversion of $S$. Having this in mind, we define
\begin{equation} \label{eq:def-delta}
\delta:={\rm det}
\begin{pmatrix}
  \frac{1}{E_1}           & -\frac{\nu_{21}}{E_2} & -\frac{\nu_{31}}{E_3} \\
  -\frac{\nu_{12}}{E_1}   & \frac{1}{E_2}         & -\frac{\nu_{32}}{E_3} \\
  -\frac{\nu_{13}}{E_1}   & -\frac{\nu_{23}}{E_2} & \frac{1}{E_3}
\end{pmatrix}
=\frac{1-\nu_{12}\nu_{21}-\nu_{13}\nu_{31}-\nu_{23}\nu_{32}-2\nu_{12}\nu_{23}\nu_{31}}{E_1
E_2 E_3}
\end{equation}
where we exploited \eqref{eq:id0} to show that $\nu_{13}\nu_{21}\nu_{32}=\nu_{12}\nu_{23}\nu_{31}$.
In this way we may write $\det(S)=\frac{\delta}{8\mu_{12}\mu_{13}\mu_{23}}$.

\subsection{Orthotropic materials with a one-dimensional reinforcement} \label{ss:p-iso}
Let us consider an orthotropic material which a one-dimensional reinforcement in the $x_1$ direction and an isotropic behavior in the $x_2$, $x_3$ variables. If we look at \eqref{eq:Hook-2}, this means that
\begin{equation} \label{eq:id1}
E_2=E_3 \, , \quad \nu_{21}=\nu_{31} \, , \quad \nu_{12}=\nu_{13}
\, , \quad \nu_{23}=\nu_{32} \, , \quad \mu_{12}=\mu_{13} \, .
\end{equation}

Moreover, we observe that, under that orthotropy condition, we have that
\begin{equation} \label{C2323}
  C_{2323}=C_{2222}-C_{2233}
\end{equation}
and hence the coefficient $\mu_{23}$ can be written as
\begin{equation} \label{eq:conjecture}
\mu_{23}=\tfrac{C_{2323}}2=\tfrac 12(C_{2222}-C_{2233}) \, .
\end{equation}

The validity of \eqref{C2323} is a well known fact and for more details we quote \cite[Section 2.4]{Orthotropic}.
It is also a well know fact that the coefficients $C_{1212}=C_{1313}$ are completely independent of the other ones as one can see from \cite[Section 2.4]{Orthotropic}.

In other words, the elastic properties of the material are uniquely determined by the following five constants
$E_1$, $E_2$, $\nu_{12}$, $\nu_{23}$, $\mu_{12}$.

Now, exploiting the identities in \eqref{eq:id0}, \eqref{eq:id1}, \eqref{eq:conjecture}, we can write the matrix $C$ in the form
{\small
\begin{equation} \label{eq:rigidity-matrix}
C=
\begin{pmatrix}
\frac{1-\nu_{23}^2}{\delta \, E_2^2} & \frac{\nu_{12}(1+\nu_{23})}{\delta \, E_1 E_2} & \frac{\nu_{12}(1+\nu_{23})}{\delta \, E_1 E_2}
& 0 & 0 & 0 \\[7pt]
\frac{\nu_{12}(1+\nu_{23})}{\delta \, E_1 E_2} & \frac{E_1-E_2 \, \nu_{12}^2}{\delta \, E_1^2 E_2} & \frac{E_1 \nu_{23}+E_2 \, \nu_{12}^2}{\delta \, E_1^2 E_2}
& 0 & 0 & 0 \\[7pt]
\frac{\nu_{12}(1+\nu_{23})}{\delta \, E_1 E_2} & \frac{E_1 \, \nu_{23}+E_2 \, \nu_{12}^2}{\delta \, E_1^2 E_2} &  \frac{E_1-E_2 \, \nu_{12}^2}{\delta \, E_1^2 E_2}
& 0 & 0 & 0 \\[7pt]
0 & 0 & 0 & 2\mu_{12} & 0 & 0 \\[7pt]
0 & 0 & 0 & 0 & 2\mu_{12} & 0 \\[7pt]
0 & 0 & 0 & 0 & 0 & \frac{E_1(1-\nu_{23})-2E_2 \, \nu_{12}^2}{\delta \, E_1^2 E_2}
\end{pmatrix} \, .
\end{equation}
}

\section{The model of a plate with a one-dimensional reinforcement} \label{s:orthotropic-plate}

This section is devoted to the construction of a model for an orthotropic plate of length $L$, width $2\ell$ and
thickness $d$ so that we may choose a coordinate system in such a way that the plate is described by
the set $(0,L)\times (-\ell,\ell)\times \left(-\frac d2,\frac d2 \right)$.

We point out that some algebraic details in the construction of the model will be omitted and for a more complete description we refer to \cite{Orthotropic}.

The usual $x,y,z$ notation will be used in place of the $x_1, x_2, x_3$ notation
used in Section \ref{s:orthotropic}.

We assume that the plate is made of an
orthotropic material with a one-dimensional reinforcement in the
$x$ direction. We assume the validity of the classical
constitutive assumptions for the displacement of a plate, see
\cite[Paragraph 11]{LaLi}:

\begin{itemize}
\item the displacement of the midway surface is only vertical and it is described by a function $u=u(x,y)$ with $(x,y)\in (0,L)\times (-\ell,\ell)$;

\item the third component of the displacement vector ${\bf u}=(u_1,u_2,u_3)$ only depends on $x$ and $y$ and with sufficient accuracy we may assume that $u_3(x,y)=u(x,y)$ for any $(x,y)\in (0,L)\times (-\ell,\ell)$;

\item the components $\sigma_{13}, \sigma_{23}, \sigma_{33}$ of the stress tensor vanish everywhere in the plate.
\end{itemize}

We now compute the elastic energy per unit of volume in a configuration corresponding to a
generic vertical displacement $u$ of the midway surface. By \eqref{eq:strain-tensor} and the above constitutive conditions we obtain
\begin{equation*}
u_{1}=-z \frac{\partial u}{\partial x} \, , \qquad u_{2}=-z \frac{\partial u}{\partial y} \, , \qquad u_3=u \, ,
\end{equation*}
and, in turn,
\begin{equation} \label{eq:e11-e22}
e_{11}=-z \frac{\partial^2 u}{\partial x^2} \, , \qquad e_{22}=-z \frac{\partial^2 u}{\partial y^2} \, , \qquad
e_{12}=-z\frac{\partial^2 u}{\partial x\partial y} \, , \qquad e_{13}=0 \, , \qquad e_{23}=0 \, ,
\end{equation}
see \cite[Paragraph 11]{LaLi} for more details. Finally, condition $\sigma_{33}=0$ combined with \eqref{eq:Hook-1}, \eqref{eq:rigidity-matrix} and \eqref{eq:e11-e22}, yields
\begin{align} \label{eq:e33}
e_{33}=\frac{\nu_{12}E_1 (1+\nu_{23})}{E_1-E_2 \nu_{12}^2} \, z \frac{\partial^2 u}{\partial x^2}+\frac{E_1 \nu_{23}+E_2 \nu_{12}^2}{E_1-E_2 \nu_{12}^2}
\, z \frac{\partial^2 u}{\partial y^2} \, .
\end{align}
Replacing \eqref{eq:e11-e22} and \eqref{eq:e33} into \eqref{eq:elastic-energy} and putting
\begin{align*}
\mathcal K=\frac{(1+\nu_{23})[E_1(1-\nu_{23})-2E_2 \nu_{12}^2]}{\delta E_1 E_2(E_1-E_2 \nu_{12}^2)}  \, ,
\end{align*}
we obtain
\begin{align*}
\mathcal E=\frac{z^2}2 \left[\frac{E_1 \mathcal K}{E_2}
\left(\frac{\partial^2 u}{\partial x^2}\right)^2+\mathcal
K\left(\frac{\partial^2 u}{\partial y^2}\right)^2 +2\nu_{12}
\mathcal K \frac{\partial^2 u}{\partial x^2}\frac{\partial^2
u}{\partial y^2}+4\mu_{12} \left(\frac{\partial^2 u}{\partial x
\partial y}\right)^2 \right] \, .
\end{align*}

Then, the total bending energy of the deformed plate in term of the
vertical displacement takes the form
\begin{equation} \label{eq:total-energy}
\mathbb E_B(u)=\frac{d^3}{24} \int_\Omega \left[\frac{E_1
\mathcal K}{E_2} \left(\frac{\partial^2 u}{\partial
x^2}\right)^2+\mathcal K\left(\frac{\partial^2 u}{\partial
y^2}\right)^2 +2\nu_{12} \mathcal K \frac{\partial^2 u}{\partial
x^2}\frac{\partial^2 u}{\partial y^2}+4\mu_{12}
\left(\frac{\partial^2 u}{\partial x
\partial y}\right)^2 \right] dxdy
\end{equation}
where we put $\Omega=(0,L)\times (-\ell,\ell)$.

We observe that by \eqref{eq:id0}, \eqref{eq:def-delta} and some computations we may write $\mathcal K$ in a more elegant way:
\begin{equation} \label{eq:write-K}
  \mathcal K=\frac{E_2}{1-\nu_{12}\nu_{21}}  \, .
\end{equation}

Looking at \eqref{eq:write-K}, we see that the elastic coefficients that completely determines the bending energy of the plate corresponding to a generic displacement $u$ are $E_1$, $E_2$, $\nu_{12}$ and $\mu_{12}$ while no dependence on the Poisson ratio $\nu_{23}$ occurs.

As pointed out in Section \ref{ss:p-iso}, the
value of the coefficient $C_{1212}$, and hence of $\mu_{12}$,
is completely independent from the others; here we assume the following condition
\begin{equation} \label{eq:ass-mu12}
\mu_{12}=\frac{\mathcal K (1-\nu_{12})}2 \, ,
\end{equation}
in complete accordance with the classical theory of isotropic plates. For more details see \cite[Chapter 1, Section 5, (5.9)]{LaLi} and \cite{Orthotropic}.

Denoting by $D^2 u$ the Hessian matrix of $u$, introducing the notation
\begin{equation*}
  D^2 u:D^2 v=u_{xx} v_{xx}+2u_{xy}v_{xy}+u_{yy}v_{yy} \quad
\text{and} \quad |D^2 u|^2=u_{xx}^2+2u_{xy}^2+v_{yy}^2 \,
\end{equation*}
defining
\begin{equation} \label{eq:write-k}
 \kappa=\frac{E_1-E_2}{E_2} \, ,
\end{equation}
and assuming \eqref{eq:ass-mu12}, we may write
\begin{equation} \label{eq:E-B}
  \mathbb E_B(u)=\frac{d^3 \mathcal K}{24} \int_\Omega \left[\nu_{12} |\Delta u|^2+(1-\nu_{12})|D^2 u|^2+\kappa\,  u_{xx}^2   \right] dxdy \, .
\end{equation}
We observe that since the plate is reinforced in the $x$ direction, we assume that
\begin{equation} \label{eq:ipotesi-Young}
E_1>E_2 \, ,
\end{equation}
which in turn implies $\kappa>0$.

In the remaining part of the paper, for the Poisson ratio $\nu_{12}$ we use the simpler notation
\begin{equation} \label{eq:write-nu}
    \nu=\nu_{12} \, .
\end{equation}
According with the theory of isotropic materials, we assume that
\begin{equation} \label{eq:ipotesi-Poisson}
0<\nu<\frac 12 \, .
\end{equation}

Let us introduce a suitable functional setting for energy $\mathbb E_B$. As explained in the introduction, our main purpose is to describe the static and dynamic behavior of the deck of a bridge by mean of a plate model. For this reason, we may assume that the deck is hinged at the two vertical edges of the rectangle $\Omega$ and free on the two horizontal edges of the same rectangle. Hence, a reasonable choice for the functional subspace of the Sobolev space $H^2(\Omega)$ is
\begin{equation} \label{eq:def-H2*}
H^2_*(\Omega):=\{w\in H^2(\Omega): w=0 \ \text{on} \ \{0,L\}\times
(-\ell,\ell)\} \, ,
\end{equation}
see \cite[Section 3]{FeGa} and \cite{Orthotropic}.

Thanks to the Intermediate Derivatives Theorem, see
\cite[Theorem 4.15]{Adams}, the space $H^2(\Omega)$ is a Hilbert
space if endowed with the scalar product
$$
(u,v)_{H^2}:=\int_\Omega\left(D^2u:D^2v+uv\right)\,
dxdy\qquad \text{for all } u,v\in H^2(\Omega) \, .
$$

On the closed subspace $H^2_*(\Omega)$ it is possible to define an
alternative scalar product naturally related to the functional $\mathbb E_B$ as explained in the next proposition.

\begin{proposition} \label{l:equivalence}
Assume \eqref{eq:ass-mu12}, \eqref{eq:ipotesi-Young} and \eqref{eq:ipotesi-Poisson}. On
the space $H^2_*(\Omega)$ the two norms
$$
u\mapsto\|u\|_{H^2}\, ,\quad
u\mapsto\|u\|_{H^2_*}:=\left\{\int_\Omega \left[\nu |\Delta u|^2+(1-\nu)|D^2 u|^2+\kappa\,  u_{xx}^2   \right] dxdy\right\}^{1/2}
$$
are equivalent. Therefore, $H^2_*(\Omega)$ is a Hilbert space when
endowed with the scalar product
\begin{equation*} 
(u,v)_{H^2_*}:=\int_\Omega \left[ \nu \Delta u \Delta v+(1-\nu)D^2 u:D^2 v+\kappa \, u_{xx}v_{xx} \right] dxdy \, .
\end{equation*}
\end{proposition}

The proof can be obtained by proceeding as in the proof of \cite[Lemma 4.1]{FeGa}.


Next, if we denote by $f$ an external vertical load per unit of
surface and if $u$ is the deflection of the plate in the vertical
direction, by \eqref{eq:E-B} and \eqref{eq:write-nu} we
have that the total energy $\mathbb E_T$ of the plate becomes
\begin{equation} \label{eq:energy-total}
\mathbb E_T(u)=\frac{d^3 \mathcal K}{24} \int_\Omega
\left(\nu|\Delta u|^2+(1-\nu)|D^2 u|^2+\kappa\,
u_{xx}^2\right) dxdy-\int_\Omega fu \, dxdy \, .
\end{equation}

A weak stationary solution $u\in H^2_*(\Omega)$ of the plate equation is a critical point of $\mathbb E_T$ so that
\begin{align} \label{eq:plate-variational}
   & \frac{d^3\mathcal K}{12} (u,v)_{H^2_*}=\int_\Omega fv \, dxdy \qquad \text{for any
} v\in H^2_*(\Omega) \, ;
\end{align}
If $u\in C^4(\overline \Omega)\cap H^2_*(\Omega)$ is a solution of the variational problem \eqref{eq:plate-variational}, by \cite[Proposition 5]{Chasman} and some calculation (see \cite{Orthotropic, FeGa} for more details), we infer
\begin{equation} \label{eq:model-plate}
\begin{cases}
\frac{d^3\mathcal K}{12} \left(\Delta^2 u+\kappa\frac{\partial^4
u}{\partial x^4}\right)=f & \qquad
\text{in } \Omega \, , \\[6pt]
u(0,y)=u_{xx}(0,y)=u(L,y)=u_{xx}(L,y)=0 & \qquad \text{for }
y\in(-\ell,\ell) \, , \\[6pt]
u_{yy}(x,\pm \ell)+\nu u_{xx}(x,\pm \ell)=0 & \qquad \text{for }
x\in (0,L) \, , \\[6pt]
u_{yyy}(x,\pm \ell)+(2-\nu)u_{xxy}(x,\pm \ell)=0 & \qquad
\text{for } x\in (0,L) \, .
\end{cases}
\end{equation}
Problem \eqref{eq:model-plate} represents the model for a plate of orthotropic material with a one-dimensional reinforcement in the $x$ direction subject to vertical load $f$ per unit of surface.

\begin{remark} \label{r:1}
We observe that recalling \eqref{eq:id0} and \eqref{eq:write-K}, the fourth order equation in  \eqref{eq:model-plate} may be written in a different way, more familiar in the theory of orthotropic plates:
{\small
\begin{equation*}
 \frac{E_1 d^3}{12(1-\nu_{12} \nu_{21})} \, \frac{\partial^4 u}{\partial x^4}
 +\frac{E_2 d^3}{6(1-\nu_{12} \nu_{21})} \, \frac{\partial^4 u}{\partial x^2\partial y^2}
 +\frac{E_2 d^3}{12(1-\nu_{12} \nu_{21})} \, \frac{\partial^4 u}{\partial y^4}=f \, ,
\end{equation*}
}
see for example \cite[Chapter 2]{DesignManual}.
\end{remark}

Let us denote by $\mathcal H(\Omega)$ the dual space of $H^2_*(\Omega)$.

In the next result we state existence, uniqueness and regularity for weak solutions of \eqref{eq:model-plate}.

\begin{theorem} \label{t:Lax-Milgram}
Assume \eqref{eq:ass-mu12}, \eqref{eq:ipotesi-Young} and \eqref{eq:ipotesi-Poisson} and let $f\in \mathcal
H(\Omega)$. Then the following conclusions hold true:

\begin{itemize}

\item[$(i)$] there exists a unique $u\in H^2_*(\Omega)$ such
that
\begin{equation*}
\frac{d^3\mathcal K}{12} (u,v)_{H^2_*}=\ _{\mathcal
H(\Omega)}\langle f,v\rangle_{H^2_*(\Omega)} \qquad \text{for any
} v\in H^2_*(\Omega) \, ;
\end{equation*}

\item[$(ii)$] $u$ is the unique minimum point of the convex functional
$$
  \mathbb E_T(u)=\frac 12 (u,u)_{H^2_*}- \ _{\mathcal
H(\Omega)}\langle f,u\rangle_{H^2_*(\Omega)}\, ;
$$

\item[$(iii)$] if $f\in W^{k,p}(\Omega)$ for some $1<p<\infty$ and $k\in \N\cup\{0\}$, where we put $W^{0,p}(\Omega):=L^p(\Omega)$, then $u\in W^{k+4,p}(\Omega)$.

\end{itemize}
\end{theorem}

Let us consider now the eigenvalue problem associated with the operator $\frac{d^3 \mathcal K}{12}\left(\Delta^2 +\kappa \frac{\partial^4 }{\partial x^4}\right)$ with the related boundary conditions:

\begin{equation} \label{eq:eigenvalue-original}
\begin{cases}
\frac{d^3 \mathcal K}{12}\left(\Delta^2 u+\kappa \frac{\partial^4 u}{\partial x^4}\right)=\lambda u & \qquad
\text{in } \Omega \, , \\[6pt]
u(0,y)=u_{xx}(0,y)=u(L,y)=u_{xx}(L,y)=0 & \qquad \text{for }
y\in(-\ell,\ell) \, , \\[6pt]
u_{yy}(x,\pm \ell)+\nu u_{xx}(x,\pm \ell)=0 & \qquad \text{for }
x\in (0,L) \, , \\[6pt]
u_{yyy}(x,\pm \ell)+(2-\nu)u_{xxy}(x,\pm \ell)=0 &  \qquad
\text{for } x\in (0,L) \, .
\end{cases}
\end{equation}
Standard spectral theory implies that the eigenvalues of \eqref{eq:eigenvalue-original} may be ordered in an increasing sequence of positive numbers diverging to $+\infty$.

After scaling we can normalize the length $L$ of the deck to $\pi$ in a such a way that \eqref{eq:eigenvalue-original} can be reduced to the new problem
\begin{equation} \label{eq:eigenvalue-normalized}
\begin{cases}
\Delta^2 u+\kappa \frac{\partial^4 u}{\partial x^4}=\lambda u & \qquad
\text{in } \widetilde\Omega:=(0,\pi)\times (-\well,\well) \, , \\[6pt]
u(0,y)=u_{xx}(0,y)=u(\pi,y)=u_{xx}(\pi,y)=0 & \qquad \text{for }
y\in(-\well,\well) \, , \\[6pt]
u_{yy}(x,\pm \well)+\nu u_{xx}(x,\pm \well)=0 & \qquad \text{for }
x\in (0,\pi) \, , \\[6pt]
u_{yyy}(x,\pm \well)+(2-\nu)u_{xxy}(x,\pm \well)=0 &  \qquad
\text{for } x\in (0,\pi) \, ,
\end{cases}
\end{equation}
where we put $\well:=\frac\pi L\, \ell$.

In this way a function $u$ is an eigenfunction of \eqref{eq:eigenvalue-original} with eigenvalue $\lambda$ if and only if the function $w(x,y)=u\left(\frac L\pi \, x,\frac L \pi \, y\right)$ is an eigenfunction of \eqref{eq:eigenvalue-normalized} with eigenvalue $\widetilde \lambda=\frac{12L^4}{\pi^4d^3 \mathcal K}\, \lambda$.


In this normalized form, problem \eqref{eq:eigenvalue-normalized} can be better compared with the eigenvalue problem for the isotropic plate introduced in \cite{FeGa}.
We denote by
\begin{equation} \label{eq:eig-plate}
  0<\lambda_1\le \lambda_2 \le \dots \le \lambda_m \le \dots
\end{equation}
the eigenvalues of \eqref{eq:eigenvalue-original} and by
\begin{equation*}
  0<\widetilde \lambda_1\le \widetilde \lambda_2 \le \dots \le \widetilde \lambda_m \le \dots
\end{equation*}
the eigenvalues of \eqref{eq:eigenvalue-normalized} thus implying the following correspondence
\begin{equation} \label{eq:eig-scaling}
  \lambda_m = \frac{\pi^4 d^3\mathcal K}{12L^4} \, \widetilde \lambda_m \qquad \text{for any } m\ge 1 \, .
\end{equation}

Since for our purposes we need a more explicit characterization of the eigenvalues and the eigenfunctions of \eqref{eq:eigenvalue-original} and \eqref{eq:eigenvalue-normalized}, we state the following

\begin{theorem} \label{t:eigenvalues}
  Assume \eqref{eq:ass-mu12}, \eqref{eq:ipotesi-Young} and \eqref{eq:ipotesi-Poisson}. Then for the eigenvalues of \eqref{eq:eigenvalue-normalized} we have:

  $(i)$ for any $m\ge1$ there
exists a sequence of eigenvalues $\widetilde\lambda_{j,m}\uparrow +\infty$
such that $\widetilde \lambda_{j,m}>(\kappa+1)m^4$ for all $j\ge1$; the corresponding
eigenfunctions are of the kind
\begin{align*}
 \left[a\cosh(\beta y)+b\sinh(\beta y)+c\cos(\gamma y)+d\sin(\gamma y)\right]\sin(mx)
\end{align*}
where $\beta=\sqrt{\sqrt{\widetilde\lambda_{j,m}-\kappa m^4}+m^2}$,
$\gamma=\sqrt{\sqrt{\widetilde\lambda_{j,m}-\kappa m^4}-m^2}$
and $a,b,c,d,$ are suitable constants depending on $m$ and $j$;

$(ii)$ denoting by $\overline s$ the unique solution of the equation $\tanh(s)=\left(\frac{\nu}{2-\nu}\right)^2 s$, if $m_*=\overline s/(\sqrt 2 \, \well)$ is an integer then $\widetilde\lambda=(\kappa+1)m_*^4$ is an eigenvalue with
corresponding eigenfunctions generated by
$$
\Big[\nu \well\sinh(\sqrt 2 m_* y)+(2-\nu)\sinh(\sqrt2 m_*\well)\, y\Big]\, \sin(m_*x)\, ;
$$

$(iii)$ for any $m\ge1$, there exists an eigenvalue
$\widetilde \lambda_m^- \in([(1-\nu)^2+\kappa]m^4,(\kappa+1)m^4)$ with corresponding
eigenfunctions generated by
\begin{align*}
& \left[\tfrac{\sqrt{\widetilde\lambda_m^- -\kappa m^4}-(1-\nu)m^2}{\cosh(\beta\well)}
\cosh(\beta y)+\tfrac{\sqrt{\widetilde\lambda_m^- -\kappa m^4}+(1-\nu)m^2}{\cosh(\gamma\well)}
\cosh(\gamma y) \right]\,
\sin(mx)\, ,
\end{align*}
where $\beta=\sqrt{m^2+\sqrt{\widetilde\lambda_m^- -\kappa m^4}}$ and $\gamma=\sqrt{m^2-\sqrt{\widetilde\lambda_m^- -\kappa m^4}}$; moreover the sequence $m\mapsto \widetilde \lambda_m^-$ is increasing;

$(iv)$ for any $m\ge1$, satisfying
\begin{equation}\label{eq:coth}
  \well m \sqrt 2\coth(\well m \sqrt 2)>\left(\tfrac{2-\nu}{\nu}\right)^2 \, ,
\end{equation}
there exists an eigenvalue $\widetilde \lambda_m^+\in(\widetilde \lambda_m^-,(\kappa+1)m^4)$ with corresponding eigenfunctions generated by
\begin{align*}
& \left[\tfrac{\sqrt{\widetilde\lambda_m^+ -\kappa m^4}-(1-\nu)m^2}{\sinh(\beta\well)}
\sinh(\beta y)+\tfrac{\sqrt{\widetilde\lambda_m^+ -\kappa m^4}+(1-\nu)m^2}{\sinh(\gamma\well)}
\sinh(\gamma y) \right]\,
\sin(mx)\, ,
\end{align*}
where $\beta=\sqrt{m^2+\sqrt{\widetilde\lambda_m^+ -\kappa m^4}}$ and $\gamma=\sqrt{m^2-\sqrt{\widetilde\lambda_m^+ -\kappa m^4}}$; moreover the sequence $m\mapsto \widetilde \lambda_m^+$ is increasing;

$(v)$ There are no eigenvalues other than the ones characterized
in $(i)-(iv)$.
\end{theorem}

\begin{remark}
  We observe that in Theorem \ref{t:eigenvalues} (i) the constants $a,b,c,d$ admit precise but implicit characterizations as solutions of quite involved algebraic equations. For this reason, in order to provide a more clear exposure, such characterizations are not reported in the statement of Theorem \ref{t:eigenvalues} but they can be found in its proof, see Section \ref{s:proof-t:eigenvalues}.
\end{remark}

\section{The behavior of orthotropic plates of small width} \label{s:small-width}

The purpose of this section is to prove some convergence results for problem \eqref{eq:model-plate} and the related eigenvalue problem \eqref{eq:eigenvalue-original}, as $\ell\to 0$. As expected, we confirm that also in the present model, the plate behaves as a one-dimensional beam when its width is relatively small if compared with its length; we recall that the convergence results for the plate equation with an external vertical load as $\ell\to 0$ were proved in \cite{FeGa} in the case of the classical isotropic plate. We also prove a spectral convergence result inspired by the arguments contained in \cite{ArFeLa, ArLa}.

If we see the plate as a parallelepiped-shaped beam $(0,L)\times (-\ell,\ell)\times \left(-\frac d2,\frac d2\right)$, we are led to the problem
\begin{equation}  \label{eq:beam-equation}
  \begin{cases}
   EI \psi''''=2\ell f \qquad \text{in } (0,L) \, , \\[6pt]
   \psi(0)=\psi(L)=\psi''(0)=\psi''(L)=0 \, ,
  \end{cases}
\end{equation}
where $f=f(x)$ represents a vertical load per unit of surface and, as a consequence, $2\ell f$ represents a vertical load per unit of length, in accordance with the classical model of an elastic hinged beam; moreover $E$ represents the Young modulus of the beam and $I=\frac{d^3 \ell}{6}=\int_{(-\ell,\ell)\times \left(-\frac d2,\frac d2\right)} z^2 \, dydz$ the moment of inertia of its cross section. Hence we have that the function $\psi$ introduced in \eqref{eq:beam-equation} solves the equation
\begin{equation} \label{eq:beam-equation-2}
  \frac{Ed^3}{12} \, \psi''''=f \qquad \text{in } (0,L) \, .
\end{equation}

In order to compare the behavior of the plate with the one of the beam for $\ell$ small, we assume that
\begin{equation} \label{eq:E-KAPPA}
  E=(\kappa+1-\nu^2) \mathcal K
\end{equation}
where $\mathcal K$ is defined in \eqref{eq:write-K}, $\kappa$ in \eqref{eq:write-k} and $\nu$ in \eqref{eq:write-nu}.

A simple computations show that \eqref{eq:E-KAPPA} implies $E=E_1$, with $E_1$ as in \eqref{eq:Hook-2}.

\begin{theorem} \label{t:conv-poisson}
  Assume \eqref{eq:ipotesi-Young}, \eqref{eq:ipotesi-Poisson} and \eqref{eq:E-KAPPA}. Let $f\in L^2(\Omega)$ be a function depending only on the $x$ variable, let $u_\ell$ be the corresponding solution of problem \eqref{eq:model-plate} and let $\psi$ be the unique solution of \eqref{eq:beam-equation}. Then the following converge holds
  \begin{equation*}
\sup_{(x,y)\in \Omega} |u_\ell(x,y)-\psi(x)|\to 0 \qquad \text{as } \ell\to 0 \, .
\end{equation*}
\end{theorem}

Our next purpose is to prove spectral convergence of the plate model to the beam model as $\ell \to 0$.


In order to emphasize the dependence of the spectrum of \eqref{eq:eigenvalue-original} from $\ell$, we denote the eigenvalues in \eqref{eq:eig-plate} by
\begin{equation} \label{eq:eig-plate-bis}
  0<\lambda_1^\ell\le \lambda_2^\ell\le \dots \lambda_m^\ell \le \dots
\end{equation}
where each eigenvalue is repeated as many times as its multiplicity.

Assuming \eqref{eq:E-KAPPA} and recalling that $I=\frac{d^3\ell}{6}$, let us consider the eigenvalue problem for the beam
\begin{equation}  \label{eq:beam-eigenvalue}
  \begin{cases}
   \frac{E d^3}{12} \psi''''=\lambda \psi \qquad \text{in } (0,L) \, , \\[6pt]
   \psi(0)=\psi(L)=\psi''(0)=\psi''(L)=0 \, ,
  \end{cases}
\end{equation}
and its eigenvalues explicitly given by
\begin{equation} \label{eq:eig-beam}
  \lambda_m^0=\frac{E d^3 \pi^4 m^4}{12 L^4} \, ,\qquad m\ge 1 \, .
\end{equation}

\begin{theorem} \label{t:spectral-convergence}
   Assume \eqref{eq:ipotesi-Young}, \eqref{eq:ipotesi-Poisson} and \eqref{eq:E-KAPPA}. For any $m\ge 1$, let $\lambda_m^\ell$ and $\lambda_m^0$ be the eigenvalues defined in \eqref{eq:eig-plate-bis} and in \eqref{eq:eig-beam} respectively. Then we have spectral convergence of problem \eqref{eq:eigenvalue-original} to problem \eqref{eq:beam-eigenvalue}, in the sense that for any $m\ge 1$ we have $\lambda_m^\ell \to \lambda_m^0$ as $\ell \to 0$.
\end{theorem}

\section{Static behavior of an orthotropic plate subject to external loads} \label{s:static}

According to Section \ref{s:orthotropic-plate}, consider a plate made of an orthotropic material with a one-dimensional
reinforcement in the $x$ direction. As already explained in Section \ref{s:orthotropic-plate}, under the validity of condition \eqref{eq:ass-mu12}, the behavior of the plate is completely determined by the three coefficients $E_1$, $E_2$ and $\nu=\nu_{12}$.

Being the deck a ``mixture'' of concrete and metal, according with \cite[p.13]{ammann}, we choose for the Poisson ratio the following value
\begin{equation*} 
\nu=0.2 \, .
\end{equation*}

The purpose of this section is to determine physically significant
values for the two Young moduli $E_1$ and $E_2$ having in mind the Tacoma Narrows
Bridge built in 1940 and the related federal report \cite{ammann}. We also take
inspiration from the paper \cite{AG17}.

The values of $E_1$ and $E_2$ for the orthotropic plate are computed in terms of the flexural and torsional rigidities
coming from the beam-rod model.

Let $L$, $\ell$, $d$ and
$\Omega=(0,L)\times (-\ell,\ell)$ be as in Section
\ref{s:orthotropic-plate}. Looking at \eqref{eq:energy-v-t}, we have that the flexural and torsional energies of a portion $(a,b)$ of the deck are given by
\begin{equation*}
  \frac{EI}{2} \int_a^b |\psi''(x)|^2 \, dx \qquad \text{and} \qquad \frac{\mathcal R_T}{2} \int_a^b |\theta'(x)|^2 \, dx
\end{equation*}
where $\psi=\psi(x)$ is the vertical displacement of the midline of the road, $\theta$ the torsion angle (see Figure \ref{f:vertical-torsional}) and $\mathcal R_T=\mu K$ the torsional rigidity.

We observe that looking at a very flexible deck, like the one of the Tacoma Narrows Bridge, the torsional constant $K$ of a real bridge appears to be of two orders of magnitude smaller than the expected one computed with formula \eqref{eq:I-K}. This apparent paradox is simply due to the fact that the cross section of a real bridge cannot be considered as a solid rectangle of steel and concrete with height $d$ and width $2\ell$, since its geometric structure is quite more complex and ``thin''  as one can see from Figure \ref{f:sezione-Tacoma} taken from \cite{Ricciardelli-Marra}.

\begin{figure}[th]
\begin{center}
  {\includegraphics[scale=0.45]{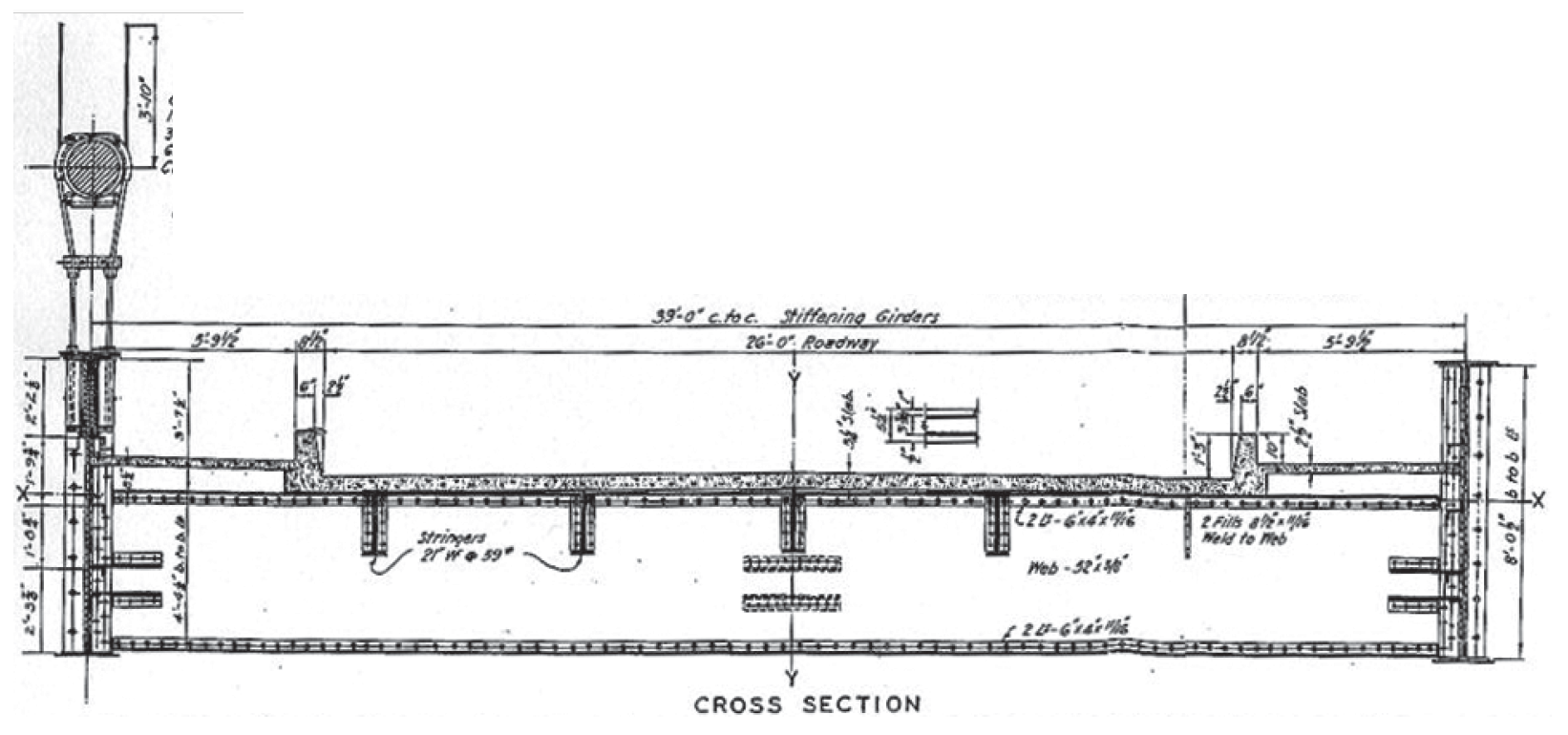}}
\caption{The cross section of the Tacoma Narrows Bridge.} \label{f:sezione-Tacoma}
\end{center}
\end{figure}





Looking at the statements of Theorems \ref{t:conv-poisson}-\ref{t:spectral-convergence}, it appears reasonable to assume here again condition \eqref{eq:E-KAPPA}.
In order to evaluate the torsional rigidity of the plate, we compare the elastic energy of a portion of plate, with the torsional energy of the corresponding portion of rod. The plate will be subject to a displacement $u$ vanishing in the midline of the road of the bridge and giving rise to a torsion; the rod will be subject to the torsion angle $\theta$ corresponding to $u$. In other words for $\rho>0$ and $\theta_0>0$ we consider
\begin{align} \label{eq:u-theta}
   &  u(x,y)=\tfrac{\theta_0}{\rho} \left(x-\tfrac L2\right)y \qquad (x,y)\in \Omega_\rho:=\left(\tfrac L2-\rho,\tfrac L2+\rho\right)
   \times (-\ell,\ell) \, ,
    \\[6pt]
  \notag &  \theta(x)=\tfrac{\theta_0}{\rho} \left(x-\tfrac L2\right) \qquad x\in \left(\tfrac L2-\rho,\tfrac L2+\rho\right) \, .
\end{align}
The function $u$ in \eqref{eq:u-theta} is obtained after carrying out the first order expansion $\tan \theta \approx \theta$ valid for small values of $\theta$.
Computing the respective elastic energies we obtain
\begin{align*}
 & \mathbb E_{B}(u)=\frac{d^3\mathcal K}{24}\int_{\Omega_\rho}\left(\nu|\Delta u|^2+(1-\nu)|D^2 u|^2+\kappa\,
   u_{xx}^2\right) dxdy=\frac{d^3\mathcal K(1-\nu)\ell}{3\rho} \, \theta_0^2  \, , \\[10pt]
 &  \mathbb E_{T}(\theta)=\frac{\mathcal R_T}{2} \int_{\frac L2-\rho}^{\frac L2 +\rho} |\theta'(x)|^2 \, dx
 =\frac{\mathcal R_T}{\rho} \, \theta_0^2 \, ,
\end{align*}
where we denoted by $\mathcal R_T$ the torsional rigidity of the rod. Equating the two elastic energies we obtain
\begin{equation} \label{eq:Kappa-R}
  \frac{d^3\mathcal K(1-\nu)\ell}{3}=\mathcal R_T \, .
\end{equation}

Recalling \eqref{eq:id0}, \eqref{eq:write-K}, \eqref{eq:E-KAPPA} and the fact that \eqref{eq:E-KAPPA} implies $E_1=E$, exploiting \eqref{eq:Kappa-R} we may represent $E_2$ in terms of $E$, $\nu$, $\mathcal R_T$ and $I=\frac{d^3\ell}{6}$:
\begin{equation} \label{eq:val-E2}
  E_2=\frac{E\mathcal R_T}{2EI(1-\nu)+\nu^2 \mathcal R_T} \, .
\end{equation}
We now look at the following values suggested in \cite{AG17} for $I$, $E$ and shear modulus $\mu$:
\begin{align*}
  & E=210 \ GPa \, , \quad I=0.15 \ m^4 \, , \quad \mu=81 \ GPa \, . 
\end{align*}
Concerning the torsional rigidity $\mathcal R_T=\mu K$, we provide a reasonable value of $K$ using well-known formulas from construction engineering. We can interpret the section of the deck of the Tacoma Narrows Bridge as made of three rectangular components like in Figure \ref{f:section}.
\begin{figure}[th]
\begin{center}
 {\includegraphics[scale=0.5]{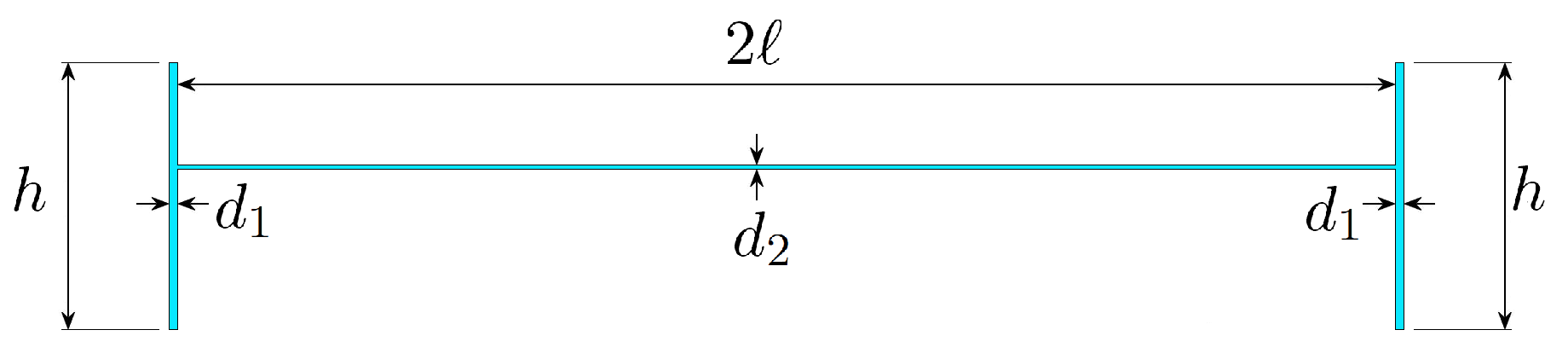}}
\caption{Cross section of the deck and computation of the rigidity constant $K$.} \label{f:section}
\end{center}
\end{figure}
The value of $K$ is then determined by the contribution of the three rectangles:
\begin{equation} \label{eq:comp-K}
  K\approx \frac 13 \, h \, d_1^3 +\frac 13 \, h \, d_1^3 +\frac 13 \cdot 2\ell \, d_2^3 = 5 \cdot 10^{-3} \ m^4
\end{equation}
where we chose $d_1=0.15 \ m$, $d_2=0.05 \ m$, $h=2 \ m$ and $2\ell=12 \ m$.
For a reference about formula \eqref{eq:comp-K} we quote \cite{Tullini}.

The value of $K$ proposed by \cite{Plaut} is $6.07 \cdot 10^{-6} \, m^4$; we believe that this value seems to be too small taking also into account the numerical experiments obtained below.

For the torsional rigidity we then obtain \begin{equation*}
  \mathcal R_T=\mu \cdot K=(8.1\cdot 10^{10})\cdot (5 \cdot 10^{-3}) \ N\cdot m^2 = 4.05 \cdot 10^8  \ N\cdot m^2
\end{equation*}
so that replacing these values in \eqref{eq:val-E2} we obtain $E_2\approx 1.687 \cdot 10^9 \ Pa$.

In our model of orthotropic plate the value of $E_2$ determines the torsional stiffness of the plate being it the coefficient of the mixed derivative $\frac{\partial^4 u}{\partial x^2 \partial y^2}$ appearing in the corresponding fourth order equation, as one can see from Remark \ref{r:1}.

Summarizing all the above assumptions, we aim to compare the behavior of an orthotropic plate satisfying
\begin{equation}\label{eq:coeff-E1-E2-nu}
  E_1=2.1\cdot 10^{11} \ Pa \, , \qquad E_2=1.687 \cdot 10^9  \ Pa \, , \qquad \nu=0.2
\end{equation}
with the one of a beam-rod system satisfying
\begin{equation}\label{eq:coeff-E-I-R}
  E=2.1\cdot 10^{11} \ Pa \, , \qquad I=0.15 \ m^4 \, , \qquad \mathcal R_T=4.05 \cdot 10^8 \ N\cdot m^2 \, .
\end{equation}
According with \cite{AG17}, in both models we assume for the length and the width of the bridge
\begin{equation}\label{eq:L-ell}
  L=853.44 \ m \, , \qquad \ell=6 \ m \, .
\end{equation}
We observe that in both models, thickness cannot be compared with the actual thickness of the deck but it has to be considered as a sort of effective thickness. In any case it is not necessary to explicit its value since it can be ``absorbed'' by the other parameters. More precisely, for the vertical displacement of the beam, $d$ can be absorbed by the moment of inertia $I$. On the other hand, for torsional deformations the corresponding equation is completely determined by the torsional rigidity $\mathcal R_T$.

Denoting by $\mathcal R=\frac{d^3 \mathcal K}{12}$ the rigidity of the plate, by \eqref{eq:write-K}, \eqref{eq:E-KAPPA}, the fact that $I=\frac{d^3 \ell}{6}$, \eqref{eq:coeff-E1-E2-nu} and \eqref{eq:coeff-E-I-R}, we infer
\begin{equation*}
  \mathcal R=\frac{d^3 E_1}{12(1+\kappa-\nu^2)}=\frac{E_1 I}{2(1+\kappa-\nu^2)\ell}=2.109 \cdot 10^7 \ N\cdot m \, .
\end{equation*}

The remaining part of this section is devoted to some numerical tests which have the purpose to evaluate the response of the two models under the action of the same vertical loads or the same moments of forces.

\medskip

{\bf A constant vertical load.} Let us consider a now a constant vertical load per unit of surface acting on the plate: in other words we consider a function $f(x,y)\equiv -f_0$ for some positive constant $f_0$.

Denoting by $u$ the corresponding solution of \eqref{eq:model-plate} and expanding in Fourier series
\begin{equation*}
  u(x,y)=\sum_{m=1}^{+\infty} Y_m(y)\sin\left(\frac{m\pi x}{L}\right) \, , \qquad (x,y)\in (0,L)\times (-\ell,\ell) \, ,
\end{equation*}
we infer for the functions $Y_m$ the following boundary value problem
\begin{equation} \label{eq:ode-y-2}
 \begin{cases}
  Y_m''''(y)-\tfrac{2m^2 \pi^2}{L^2}\, Y_m''(y)+\tfrac{(\kappa+1)m^4\pi^4}{L^4} Y_m(y)=-\frac{48f_0}{d^3\mathcal K\pi m} \, , \\[7pt]
  Y_m''(\pm \ell)-\nu \frac{m^2\pi^2}{L^2}\, Y_m(\pm \ell)=0 \, , \\[7pt]
  Y_m'''(\pm \ell)+(\nu-2)\frac{m^2\pi^2}{L^2}\, Y_m'(\pm \ell)=0 \, ,
 \end{cases}
\end{equation}
when $m$ is odd and $Y_m\equiv 0$ when $m$ is even. We point out that problem \eqref{eq:ode-y-2} can be solved explicitly, see the proof of Theorem \ref{t:conv-poisson}, but for simplicity we omit here the explicit solution favouring the graphical approach.

With a load $f_0=10^2 N/m^2$ we compute numerically the vertical displacement of the plate in the midline of the road. The graph of $u(x,0)$ is shown in Figure \ref{f:carico-costante} on the left.

The corresponding vertical load per unit of length is given by $F(x)\equiv -f_0 \cdot 2\ell$. Inserting $F$ in the first equation of \eqref{eq:system-u-theta} and assuming the Navier boundary conditions $\psi(0)=\psi(L)=\psi''(0)=\psi''(L)=0$ we obtain the explicit solution
\begin{equation*}
  \psi(x)=-\frac{\ell f_0}{12EI} (x^4-2Lx^3+L^3 x) \qquad \text{for any } x\in (0,L) \, .
\end{equation*}
The graph of $\psi$ is shown in Figure \ref{f:carico-costante} on the right.

\begin{figure}[th]
\begin{center}
   {\includegraphics[scale=0.295]{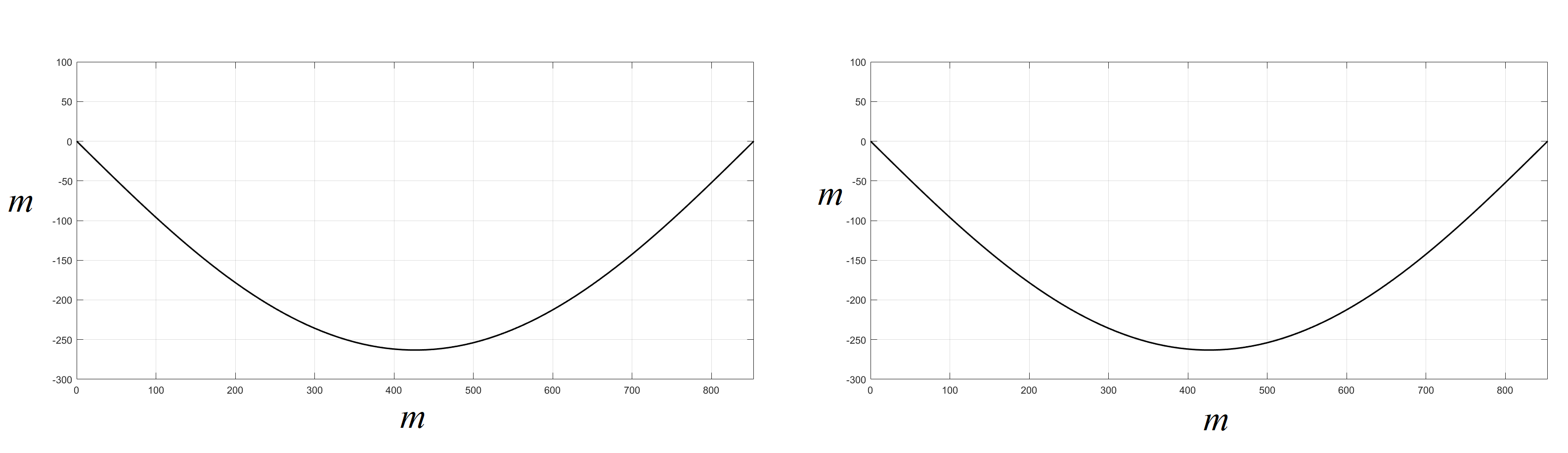}}
\caption{Downward vertical load of $1200 \ N/m$: $u(\cdot,0)$ on the left and $\psi$ on the right expressed in meters.} \label{f:carico-costante}
\end{center}
\end{figure}

We observe that the graphs of the two functions $x\mapsto u(x,0)$ and $x\mapsto \psi(x)$ are essentially indistinguishable; we specify that their difference is of the order of $10^{-3}$.

One may wonder about the fact that the vertical load per unit of length we have chosen in the numerical simulation is much smaller if compared with the weight per unit of length of the deck: 1200 $N/m$ versus $7.06\cdot 10^4 \ N/m$. Despite the quite small load, we observe in Figure \ref{f:carico-costante} a downward vertical displacement at $x=L/2$ of more than $250 \ m$. This happens because in our models for the deck, we only consider the contribution of bending and we neglect the contribution of stretching (see \cite{AlBeGa, Ba, BaBeFeGa, Di, GhiGo, Gro, Ya} for models describing stretching of the deck); with a so large deflection as the one observed in Figure \ref{f:carico-costante}, stretching cannot be neglected but in a real suspension bridge such deflections of the deck are impossible due to the presence of the cables, so that in this last situation one can only focus on the contribution of bending.
We point out that the present numerical simulation has the only purpose to test the response given by the two models and to compare them.

\medskip

{\bf A vertical load of the type $\mathbf{f(x,y)=-f_0 \, sin\left(\frac{\pi x}{L}\right)}$.} As in the previous example we choose $f_0=10^2 \, N/m^2$ and we put
\begin{equation} \label{eq:carico-f0sin}
 f(x,y)=-f_0 \sin\left(\frac{\pi x}{L}\right) \qquad \text{and}  \qquad
 F(x)=-2\ell f_0 \sin\left(\frac{\pi x}{L}\right) \, .
\end{equation}
Proceeding as in the previous case we infer
\begin{align*}
  & u(x,y)=Y(y)\sin\left(\frac{\pi x}{L}\right) \, , \qquad (x,y)\in (0,L)\times (-\ell,\ell) \, ,  \\[7pt]
  & \psi(x)=-\frac{2\ell L^4 f_0}{EI\pi^4} \, \sin\left(\frac{\pi x}{L}\right) \, , \qquad x\in (0,L) \, .
\end{align*}
where $Y$ solves
\begin{equation} \label{eq:ode-y-3}
 \begin{cases}
  Y''''(y)-\tfrac{2 \pi^2}{L^2}\, Y''(y)+\tfrac{(\kappa+1)\pi^4}{L^4} Y(y)=-\frac{12f_0}{d^3\mathcal K} \, , \\[7pt]
  Y_m''(\pm \ell)-\nu \frac{\pi^2}{L^2} \, Y_m(\pm \ell)=0 \, , \\[7pt]
  Y_m'''(\pm \ell)+(\nu-2) \frac{\pi^2}{L^2} \, Y_m'(\pm \ell)=0 \, .
 \end{cases}
\end{equation}
Also in this case, for simplicity we omit the explicit representation of the function $Y$ solving \eqref{eq:ode-y-3}.
The graphs of the functions $x\mapsto u(x,0)$ and $x\mapsto \psi(x)$ are shown in Figure \ref{f:carico-f0sin}.
Their difference is also in this case very small and more precisely about of the order of $10^{-3}$.

\begin{figure}[th]
\begin{center}
   {\includegraphics[scale=0.28]{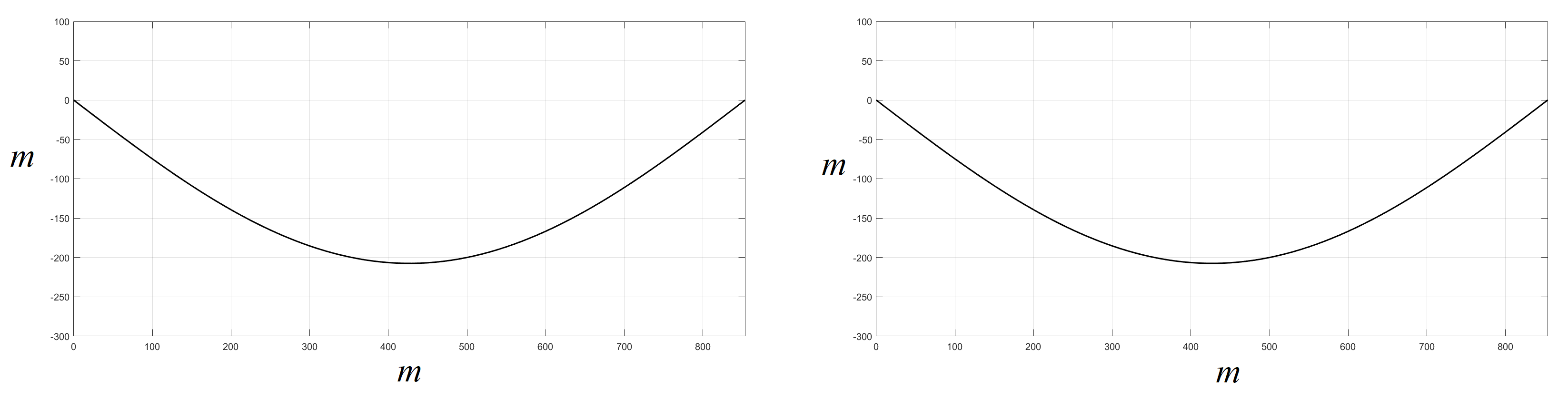}}
\caption{Downward vertical load \eqref{eq:carico-f0sin}: the displacement $u(\cdot,0)$ on the left and of $\psi$ on the right expressed in meters.} \label{f:carico-f0sin}
\end{center}
\end{figure}

\vfill

\eject

{\bf A vertical load generating torsion.} Let us introduce the constant $\tau_0=75 \, N/m^3$ and let us define
\begin{equation}\label{eq:load-torsion}
  f(x,y)=\tau_0 y \sin\left(\frac{2\pi x}{L}\right) \qquad \text{and} \qquad M(x)=\int_{-\ell}^{\ell} y f(x,y) \, dy
  =\frac 23 \, \tau_0 \ell^3 \sin\left(\frac{2 \pi x}{L} \right) \, .
\end{equation}
This time we consider the solution $\theta=\theta(x)$ of the second equation in \eqref{eq:system-u-theta} corresponding to this particular choice of $M$ which denotes the moment of forces per unit of length.

After calculation we deduce that
\begin{align*}
   & u(x,y)=Y(y)\sin\left(\frac{2\pi x}{L}\right) \, , \qquad (x,y)\in (0,L)\times (-\ell,\ell) \, , \\[7pt]
   & \theta(x)=\frac{\tau_0\, \ell^3 L^2}{6\pi^2 \mathcal R_T} \, \sin\left(\frac{2\pi x}{L}\right) \, ,
   \qquad x\in (0,L) \, .
\end{align*}
where $Y$ solves
\begin{equation} \label{eq:ode-y-4}
 \begin{cases}
  Y''''(y)-\tfrac{8 \pi^2}{L^2}\, Y''(y)+\tfrac{16(\kappa+1)\pi^4}{L^4} Y(y)=\frac{12 \tau_0}{d^3\mathcal K} \, y \, , \\[7pt]
  Y''(\pm \ell)-\nu  \tfrac{4 \pi^2}{L^2} Y(\pm \ell)=0 \, , \\[7pt]
  Y'''(\pm \ell)+(\nu-2) \tfrac{4 \pi^2}{L^2} Y'(\pm \ell)=0 \, .
 \end{cases}
\end{equation}

The graphs of the torsion angles $x \mapsto \theta(x)$ and $x \mapsto \arctan(u(x,\ell)/\ell)$ measured in radians can be found in Figure \ref{f:theta-u}.

\begin{figure}[th]
\begin{center}
    {\includegraphics[scale=0.29]{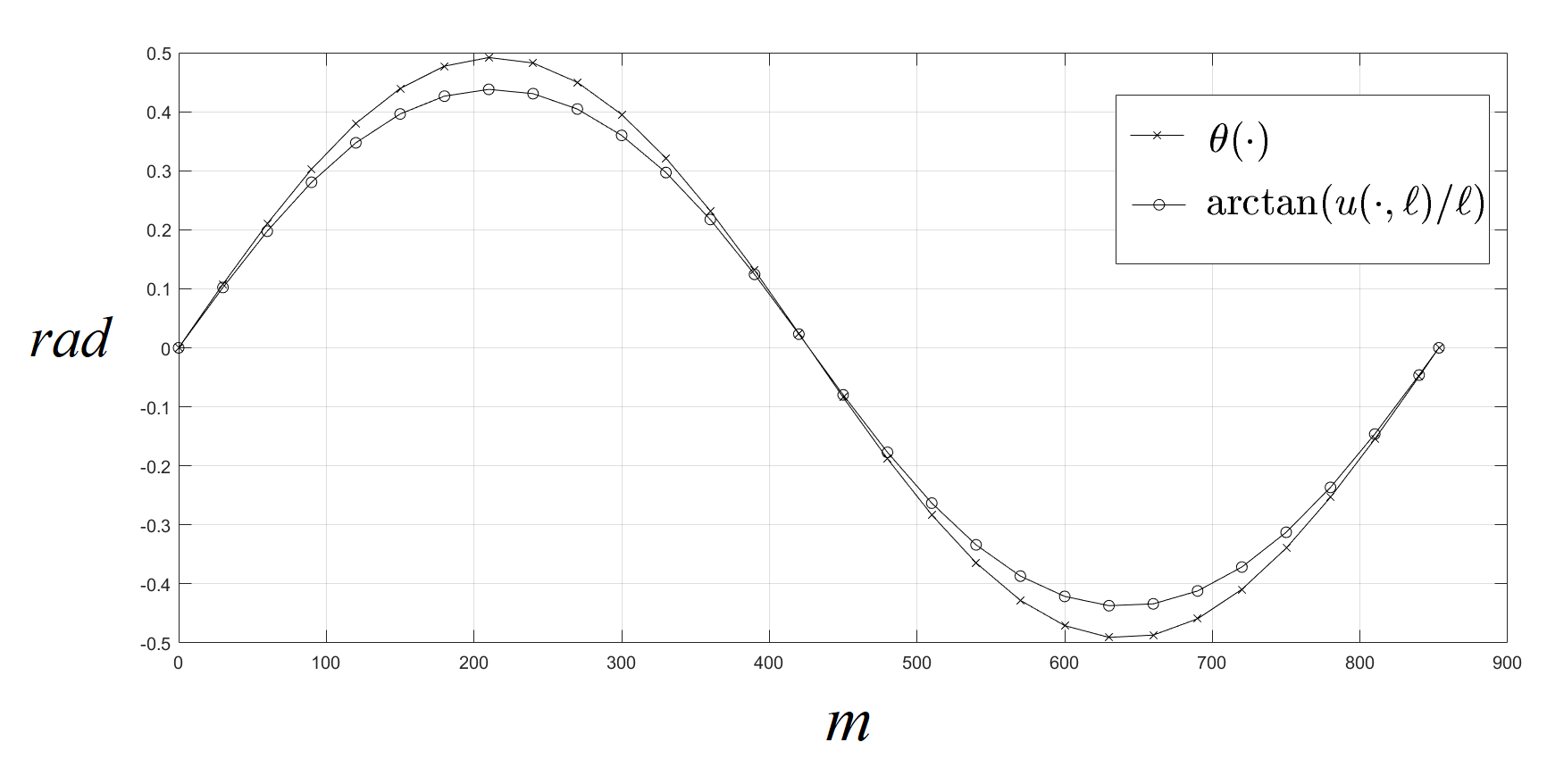}}
\caption{Functions $\theta(\cdot)$ and $\arctan(u(\cdot,\ell)/\ell)$} \label{f:theta-u}
\end{center}
\end{figure}

We observe that the difference between the two models in the measure of the torsion angle is more significant if compared
with the difference of the vertical displacements in the midline of the deck.

Taking higher values of the torsional rigidity $\mathcal R_T$, one observes a smaller difference between the torsion angles.

\section{Eigenvalue and vibration frequencies} \label{s:frequence}

This section is devoted to the comparison between the orthotropic plate and the beam-rod from the point of view of their respective natural frequencies of vibration and hence of their spectra.

In doing this, we fix the values of the parameters of the deck assuming \eqref{eq:coeff-E1-E2-nu}, \eqref{eq:coeff-E-I-R}, \eqref{eq:L-ell}.

It is well known that in order to obtain the natural frequencies of vibration from the equation of motion of some model, one can consider a \textit{stationary wave solution} in the form $u({\bm x},t)=\sin(\omega t) U({\bm x})$ where $\omega$ represents the angular velocity. The frequency is then obtained from $\omega$ by dividing it by $2\pi$.

Let us consider the equation of motion for a free orthotropic plate:

\begin{equation*}
  \frac{M}{2\ell} \, \frac{\partial^2 u}{\partial t^2}+\frac{d^3 \mathcal K}{12}
   \left(\Delta^2 u+\kappa \, \frac{\partial^4 u}{\partial x^4}\right)=0
\end{equation*}
where $M$ is the mass linear density of the deck. According to \eqref{eq:eig-plate}, the natural frequencies of vibration are given by
\begin{equation} \label{eq:nu-m}
  \nu_m=\frac{1}{\pi} \sqrt{\frac{\ell \lambda_m}{2M}} \, , \qquad m\in \N, \ m\ge 1 \, .
\end{equation}

The equation of motion for vertical oscillations is given by

\begin{equation*}
  M \frac{\partial^2 \psi}{\partial t^2}+EI \frac{\partial^4 \psi}{\partial x^4}=0
\end{equation*}
and observing that the eigenvalues of the operator $EI \partial^4_x$ can be obtained multiplying by $2\ell$ the values in
\eqref{eq:eig-beam}, the natural frequencies for vertical vibration are
\begin{equation*}
  \mu_m=\sqrt{\frac{EI \pi^2}{4ML^4}} \, m^2  \, , \qquad m\in \N, \ m\ge 1 \, .
\end{equation*}

The equation of motion for torsion is given by
\begin{equation*}
  \frac{M \ell^2}{3} \frac{\partial^2 \theta}{\partial t^2}-\mathcal R_T \frac{\partial^2 \theta}{\partial x^2}=0
\end{equation*}
and observing that the eigenvalues of the operator $-\mathcal R_T \partial^2_x$ are given by
$\frac{m^2 \pi^2 \mathcal R_T}{L^2}$, the natural frequencies for torsional vibration are
$$
   \tau_m=\sqrt{\frac{3 \mathcal R_T}{4M\ell^2 L^2}} \, m  \, , \qquad m\in \N, \ m\ge 1 \, .
$$

In the next tables we compare the values of the frequencies coming from the plate model and beam-rod model assigning to mass linear density $M$ the following value taken from \cite{AG17}:
$$
    M=7198 \ kg/m \, .
$$

Since frequencies $\nu_m$ come from both vertical and torsional eigenvalues, we introduce a suitable notation to distinguish these two kinds of eigenvalues.

We introduce the following eigenvalues of the plate
\begin{equation*}
  \lambda_m^{{\rm vert}}=\frac{\pi^4 d^3 \mathcal K}{12 L^4} \, \widetilde \lambda_m^- \, , \qquad m\in \N,  \ m\ge 1
\end{equation*}
where $\widetilde \lambda_m^-$ are the eigenvalues of the rescaled problem introduced in Theorem \ref{t:eigenvalues} (iii), see \eqref{eq:eig-scaling}.

In the same way we consider the \textit{torsional eigenvalues}

\begin{equation*}
  \lambda_m^{{\rm tors}}=\frac{\pi^4 d^3 \mathcal K}{12 L^4} \, \widetilde \eta_{m} \, , \qquad m\in \N,  \ m\ge 1
\end{equation*}
with $\widetilde \eta_{m}=\widetilde \lambda_{j,m}$ where the number $\widetilde \lambda_{j,m}$ introduced in Theorem \ref{t:eigenvalues} (i) is the least value for which system \eqref{eq:system-10} admits a nontrivial solution.
This choice of $\widetilde \eta_m$ produces an eigenfunction which is odd with respect to the $y$ variable.

Correspondingly to $\lambda_m^{{\rm vert}}$ and $\lambda_m^{{\rm tors}}$, we may define $\nu_m^{{\rm vert}}$ and $\nu_m^{{\rm tors}}$ through formula \eqref{eq:nu-m}.

\begin{table}
  \begin{center}
    \begin{tabular}{|c|c|c|c|c|c|c|c|c|c|c|}
      \hline
         $m$                    & $1$       & $2$       & $3$      & $4$       & $5$       & $6$      & $7$ & $8$ & $9$ & $10$ \\
      \hline
         $\nu_m^{{\rm vert}}$   & $0.0045$  & $0.0180$  & $0.0406$ & $0.0722$  & $0.1128$ & $0.1624$ & $0.2211$ &
         $0.2887$ & $0.3654$    & $0.4512$ \\
      \hline
         $\mu_m$                & $0.0045$  & $0.0180$  & $0.0406$ & $0.0722$  & $0.1128$ & $0.1624$ & $0.2211$ &
         $0.2887$ & $0.3654$    & $0.4512$ \\
      \hline
    \end{tabular}
    \bigskip
  \caption{First ten frequencies corresponding to vertical oscillations measured in $Hz$}  \label{t:1}
 \end{center}
\end{table}

In Table \ref{t:1} we consider the first ten vertical eigenvalues. From the table we observe that there is no difference between the frequencies $\nu_m^{{\rm vert}}$ and $\mu_m$ coming from the two approaches since their difference is of the order of $10^{-7} \, Hz$.

On the other hand, a more sensible difference can be observed in Table \ref{t:2} between the torsional frequencies $\nu_m^{{\rm tors}}$ and $\tau_m$ for $m=1,\dots, 8$. However this difference remains relatively small for the first four values and becomes larger for $m=5,6,7,8$. We believe that the most significant oscillation modes that can be observed in a real bridge are the ones with $m=1$ and $m=2$, see \cite{Tacoma}.

\begin{table}
  \begin{center}
    \begin{tabular}{|c|c|c|c|c|c|c|c|c|}
      \hline
         $m$                    & $1$       & $2$       & $3$      & $4$       & $5$       & $6$  & $7$  & $8$  \\
      \hline
         $\nu_m^{{\rm tors}}$   & $0.0404$  & $0.0822$  & $0.1270$ & $0.1760$  & $0.2301$  & $0.2904$
                                & $0.3574$  & $0.4317$    \\
      \hline
         $\tau_m$               & $0.0401$  & $0.0802$  & $0.1204$ & $0.1605$  & $0.2006$  & $0.2407$
                                & $0.2808$  & $0.3209$  \\
      \hline
    \end{tabular}
    \bigskip
  \caption{First six frequencies corresponding to torsional oscillations measured in $Hz$}  \label{t:2}
 \end{center}
\end{table}

\section{Future developments} \label{s:develop}

A further step, looking at the present article as a starting point, could be the formulation of a complete model of suspension bridge where the behavior of its deck is described by an orthotropic plate.

In \cite{BeFeGa,FeGa} one can find a preliminary version of a model of a complete suspension bridge in which the classical isotropic plate is used to describe the behavior of the deck; the action of cables and hangers was described by a polynomial nonlinearity depending only on the displacement $u$ of the deck. In that model, the action of cables and hangers was confined to a restricted region $\omega$ of the rectangle $\Omega=(0,L)\times (-\ell,\ell)$ consisting of two thin parallel strips adjacent to the two free edges of the plate:
\begin{equation} \label{eq:omega}
   \omega:=(0,L)\times [(-\ell,-\ell+\eps) \cup (\ell-\eps,\ell)] \, .
\end{equation}
Inspired by \cite{ArLaVaMa,AG15,AG17}, a more realistic model can be obtained coupling the equation of the orthotropic plate with the equations of cables, thus obtaining a system of three coupled equations, one for the plate and one for each of the two cables. The coupling of the three equations is caused by the presence of the hangers connecting the deck to the two cables; their action can be described by a suitable nonlinearity depending on the position of deck and cables.

According with \cite{AG15, AG17}, let us denote by $p_1=p_1(x,t)$ and $p_2(x,t)$ the displacements of the two cables from their rest position determined by being subject to their own weight and to the weight of the deck. As in previous sections of the present article, we denote by $u=u(x,y,t)$ the displacement of the deck from its rest position.

We also denote by $m$ the mass linear density of the cables and by $M$ the mass linear density of the deck.

A first attempt can bring to a system in the form
\begin{equation} \label{eq:sistemone}
  \begin{cases}
     m \xi(x) \frac{\partial^2 p_1}{\partial t^2}-\frac{H_0}{(\xi(x))^2} \, \frac{\partial^2 p_1}{\partial x^2}
     =f_1\left(x,p_1,\frac{\partial p_1}{\partial x}\right)+F(u(\cdot,\ell)-p_1)  \\[7pt]
      m \xi(x) \frac{\partial^2 p_2}{\partial t^2}-\frac{H_0}{(\xi(x))^2} \, \frac{\partial^2 p_2}{\partial x^2}
     =f_2\left(x,p_2,\frac{\partial p_2}{\partial x}\right)+F(u(\cdot,-\ell)-p_2)  \\[7pt]
     \frac{M}{2\ell} \frac{\partial^2 u}{\partial t^2}+\frac{d^3 \mathcal K}{12}\left(\Delta^2 u+\kappa
     \frac{\partial^4 u}{\partial x^4}\right)=-F(u(\cdot,\ell)-p_1)-F(u(\cdot,-\ell)-p_2)
  \end{cases}
\end{equation}
where $s=s(x)$ is the configuration of cables at rest, $\xi(x)=\sqrt{1+(s'(x))^2}$ is the local length of cables at rest, $H_0$ is the horizontal component of the tension of cables and $f_1, f_2, F$ are suitable nonlinearities to be determined in dependence of the accuracy one aims to achieve in the model.

As one can see from \cite{AG15,AG17}, further refinements in the model of the bridge can produce some additional integral terms: this happens for example when we want to consider the additional tension due to the increment of length of cables.

Following \cite{FeGa}, a possible alternative to the third equation in \eqref{eq:sistemone} could be
$$
  \tfrac{M}{2\ell} \tfrac{\partial^2 u}{\partial t^2}+\tfrac{d^3 \mathcal K}{12}\left(\Delta^2 u+\kappa
     \tfrac{\partial^4 u}{\partial x^4}\right)=\Upsilon_\omega (x,y) [-F(u-p_1)-F(u-p_2)] \, ,
$$
where $\Upsilon_\omega$ is the characteristic function of the set $\omega$ defined in \eqref{eq:omega}, thus removing from the equation the distributional terms $u(\cdot,\ell)$ and $u(\cdot,-\ell)$.

For a system like \eqref{eq:sistemone} and its variants, the first step is to study well-posedness of the associated initial value problem.
This may be done by using classical method like the Galerkin method which produces as a byproduct approximate solutions to be
used in the numerical simulation.

In the paper \cite{BeFeGa}, we studied an instability phenomenon that may occur in the behavior of a suspension bridge consisting in a sudden appearance of torsional oscillations of the deck triggered by vertical oscillations with sufficiently large energy. This kind of phenomenon seems to be not so infrequent in the history of suspension bridges: we quote for example the Tacoma Narrows Bridge collapsed in 1940 after few month from its inauguration, the Brighton Chain Pier erected in 1823 and collapsed in 1836 and the Matukituki Suspension Footbridge in New collapsed in 1977.
For a more detailed explanations we suggest to read the introduction of \cite{BeFeGa}.

As mentioned above, in \cite{BeFeGa} we followed a model, previously developed in \cite{FeGa}, in which the oscillations of the deck are described by an isotropic plate equation in which the combined action of cables and hangers is described by a suitable nonlinear term. The oscillations modes, vertical or torsional, were characterized by mean of the eigenfunctions of the plate equation.

The problem was to study how much vertical oscillations mode were prone to trigger torsional oscillation modes; roughly speaking, we called stable with respect to a specific torsional oscillation mode, a vertical oscillation mode which does not transfer its energy to that torsional oscillation mode and unstable otherwise. We analytically proved that below suitable energy thresholds, vertical modes are stable. On the other hand, numerical evidence suggested that above these energy thresholds vertical modes become unstable with respect to torsional modes. We identified some vertical modes which are more prone to transfer their energy to the torsional modes.

Looking at those results, we believe that it should be meaningful to follow the path drawn in \cite{BeFeGa} and to test the behavior of the suspension bridge by using the more refined model \eqref{eq:sistemone}.

We wonder if the results obtained in \cite{BeFeGa} will be confirmed in this new setting.


\section{Proof of Theorem \ref{t:Lax-Milgram}} The proofs of (i) and (ii) follows immediately from the Lax-Milgram Theorem and Proposition \ref{l:equivalence}.

It remains to prove the regularity result in (iii). We proceed by applying the elliptic regularity results by \cite{adn}, see also \cite{gazgruswe}. In order to overcome the lack of smoothness of $\partial\Omega$ one can proceed as in \cite[Lemma 4.2]{FeGa} with an odd extension argument with respect to the vertical edges.
The only one thing that remains to prove is the validity of {\it complementing conditions} by \cite{adn}.

Let us denote by ${\bf n}=(n_1,n_2)$ the unit external normal vector to the boundary and by ${\bm \tau}=(\tau_1,\tau_2)\neq (0,0)$ any tangential vector to the boundary.

Let $L(x,y)=(\kappa+1)x^4+2x^2 y^2+y^4$ be the \textit{characteristic polynomial} corresponding to the differential operator $\Delta^2+\kappa \frac{\partial^4}{\partial x^4}$, let $P(t)=L({\bm \tau}+t \bf n)$ and let $t_1^+$, $t_2^+$ be the two complex roots of $P$ having positive imaginary part. By direct computation one can check that $t_1^+=\beta+i\gamma$, $t_2^+=-\beta+i\gamma$ where
\begin{align*}
   &   \beta=|\tau_1| \sqrt{\frac{\sqrt{\kappa+1}-1}{2}}\, , \quad   \gamma=|\tau_1| \sqrt{\frac{\sqrt{\kappa+1}+1}{2}}
     \qquad \text{on horizontal edges,}  \\[10pt]
   & \beta=|\tau_2| \sqrt{\frac{\sqrt{\kappa+1}-1}{2(\kappa+1)}}\, , \quad   \gamma=|\tau_2| \sqrt{\frac{\sqrt{\kappa+1}+1}{2(\kappa+1)}}
     \qquad \text{on vertical edges,}
\end{align*}
see \cite[Page 626]{adn} for more details on the definition of complementing conditions.

Let us define now, the polynomial $B(t)=(t-t_1^+)(t-t_2^+)$. We now prove separately the validity of the complementing conditions on horizontal and vertical edges.

\medskip

{\bf Horizontal edges.} On these two edges we have that $n_1=0$, $n_2=\pm 1$ and $\tau_2=0$. Let us define the characteristic polynomials $L_1(x,y)=y^2+\nu x^2$, $L_2(x,y)=y^3+(2-\nu)x^2y$, corresponding to the boundary operators $\frac{\partial^2}{\partial y^2}+\nu \frac{\partial^2}{\partial x^2}$ and $\frac{\partial^3}{\partial y^3}+(2-\nu)\frac{\partial^3}{\partial x^2 \partial y}$ respectively. Now, let
\begin{align*}
   & B_1(t)=L_1({\bm \tau}+t{\bf n})=t^2+\nu \tau_1^2 \, , \quad
     B_2(t)=L_2({\bm \tau}+t{\bf n})=n_2 \left[t^3+(2-\nu)\tau_1^2 \, t\right] \, .
\end{align*}

We have to prove the linear independence of $B_1(t)$ and $B_2(t)$ \ mod \ $B(t)$, i.e. if $a\in \mathbb C$ and $b\in \mathbb C$ are such that $aB_1(t)+bB_2(t)\equiv 0$ \ mod \ $B(t)$ then $a=b=0$. Dividing the polynomials $aB_1(t)+bB_2(t)$ by $B(t)$, we obtain the following remainder polynomial
$$
R(t)=\left\{\left[(2-\nu)\tau_1^2+\beta^2+\gamma^2\right]n_2 b+2i\gamma(a+2i\gamma n_2 b)\right\}t+a\nu \tau_1^2
+(a+2i\gamma n_2 b)(\beta^2+\gamma^2)
$$
which, by assumption, is the null polynomial. Hence, equating its coefficients to zero, we obtain a homogeneous system in the unknowns $a$ and $b$. Letting $M$ the matrix of the coefficients of that system and recalling that $|n_2|=|{\bf n}|=1$, we have that
$$
n_2 \cdot {\rm det}(M)=-4\gamma^2(\beta^2+\gamma^2)-\left[(2-\nu)\tau_1^2+\beta^2-3\gamma^2\right](\nu\tau_1^2+\beta^2+\gamma^2)
=:g(\tau_1) \, .
$$
By elementary calculus, we see that the function $g$ admits a unique stationary point at $\tau_1=0$ provided that \eqref{eq:ipotesi-Poisson} holds true, thus showing that $g$ achieves its maximum at $\tau_1=0$. But $g(0)=-(\beta^2+\gamma^2)^2<0$ and hence $g(\tau_1)<0$ for any $\tau_1\in \R$ and hence ${\rm det}(M)\neq 0$ for any $\tau_1\in \R$. This proves that $a=b=0$ and hence the linear independence of $B_1(t)$ and $B_2(t)$ mod $B(t)$.

\medskip

{\bf Vertical edges.} On these two edges we have that $n_1=\pm 1$, $n_2=0$ and $\tau_1=0$. Proceeding similarly to the case of the horizontal edges, we obtain the polynomials $L_1(x,y)=1$ and $L_2(x,y)=x^2$ and
$$
     B_1(t)=L_1({\bm \tau}+t{\bf n})=1 \, , \quad  B_2(t)=L_2({\bm \tau}+t{\bf n})=t^2 \, .
$$
Dividing the polynomials $aB_1(t)+bB_2(t)$ by $B(t)$, we obtain the following remainder polynomial
$$
   R(t)=2i\gamma b t+a+(\beta^2+\gamma^2)b \, .
$$
Equating $R$ to the null polynomial we immediately obtain $a=b=0$ and hence the linear independence of $B_1(t)$ and $B_2(t)$ mod $B(t)$.

\section{Proof of Theorem \ref{t:eigenvalues}} \label{s:proof-t:eigenvalues}

We look for solutions of \eqref{eq:eigenvalue-normalized} in the form
\begin{equation*}
u(x,y)=\sum_{m=1}^{+\infty} h_m(y)\sin(mx) \qquad \text{for
}(x,y)\in (0,\pi)\times(-\well,\well) \, .
\end{equation*}
Since $u$ is smooth on the closure of $\widetilde \Omega$, as one can deduce by a bootstrap argument based on Theorem \ref{t:Lax-Milgram}(iii), the functions $h_m$ are smooth and they solve the equation
\begin{equation} \label{eq:ode-y}
h_m''''(y)-2m^2 h_m''(y)+[(1+\kappa)m^4-\lambda]h_m(y)=0
\end{equation}
and they are subject to the following boundary conditions
\begin{equation} \label{eq:boundary-hm}
h_m''(\pm \well)-\nu m^2h_m(\pm\well)=0 \, , \qquad
h_m'''(\pm\well)+(\nu-2)m^2 h_m'(\pm\well)=0 \, .
\end{equation}
The characteristic equation associated with \eqref{eq:ode-y} becomes
\begin{equation} \label{eq:char}
\alpha^4-2m^2\alpha^2+[(1+\kappa)m^4-\lambda]=0 \, .
\end{equation}
We have to distinguish several cases related to the structure of solutions of \eqref{eq:char}.

\medskip

$\bullet$ {\bf The case $0<\lambda<\kappa \, m^4$.}  We have that \eqref{eq:char} admits four solutions in the form
\begin{equation*}
\alpha=\beta\pm i\gamma \quad \text{or} \quad \alpha=-\beta\pm i\gamma
\end{equation*}
with
\begin{equation*}
\beta:=\sqrt{\frac{\sqrt{(\kappa+1)m^4-\lambda}+m^2}2} \quad \text{and} \quad \gamma:=\sqrt{\frac{\sqrt{(\kappa+1)m^4-\lambda}-m^2}2} \, .
\end{equation*}
One can easily verify that
\begin{equation} \label{eq:rel-beta-gamma}
\beta^2-\gamma^2=m^2 \qquad \text{and} \qquad 2\beta\gamma=\sqrt{\kappa m^4-\lambda} \, .
\end{equation}
Hence, the general solution of \eqref{eq:ode-y} is in the form
\begin{equation*}
h_m(y)=a\cosh(\beta y)\cos(\gamma y)+b\sinh(\beta y)\cos(\gamma y)+c\cosh(\beta y)\sin(\gamma y)+d\sinh(\beta y)\sin(\gamma y) \, .
\end{equation*}
The derivatives of $h_m$ until order three are given by
\begin{align*}
 h_m'(y)& =(b\beta+c\gamma)\cosh(\beta y)\cos(\gamma y)+(a\beta+d\gamma)\sinh(\beta y)\cos(\gamma y) \\[7pt]
& \qquad +(-a\gamma+d\beta)\cosh(\beta y)\sin(\gamma y)
+(-b\gamma+c\beta)\sinh(\beta y)\sin(\gamma y) \, , \\[7pt]
h_m''(y) & =(a\beta^2+2d\beta\gamma-a\gamma^2)\cosh(\beta y)\cos(\gamma y)+(b\beta^2+2c\beta\gamma-b\gamma^2)\sinh(\beta y)\cos(\gamma y) \\[7pt]
& \qquad +(c\beta^2-2b\beta \gamma-c\gamma^2)\cosh(\beta y)\sin(\gamma y)+(d\beta^2-2a\beta\gamma-d\gamma^2)\sinh(\beta y)\sin(\gamma y) \, , \\[7pt]
h_m'''(y) & =(b\beta^3+3c\beta^2\gamma-3b\beta\gamma^2-c\gamma^3)\cosh(\beta y)\cos(\gamma y) \\[7pt]
& \qquad +(a\beta^3+3d\beta^2\gamma-3a\beta\gamma^2-d\gamma^3)\sinh(\beta y)\cos(\gamma y) \\[7pt]
& \qquad +(d\beta^3-3a\beta^2\gamma-3d\beta\gamma^2+a\gamma^3)\cosh(\beta y)\sin(\gamma y) \\[7pt]
& \qquad +(c\beta^3-3b\beta^2\gamma-3c\beta\gamma^2+b\gamma^3)\sinh(\beta y)\sin(\gamma y) \, .
\end{align*}
Imposing the boundary conditions \eqref{eq:boundary-hm} we obtain the following two systems
{\small
\begin{align} \label{eq:system-1}
\begin{cases}
\left[\left(\beta^2-\gamma^2-\nu m^2\right)\cosh(\beta\well)\cos(\gamma\well)-2\beta\gamma\sinh(\beta\well)\sin(\gamma\well)
\right]a \\[6pt]
 \qquad\qquad +\left[\left(\beta^2-\gamma^2-\nu m^2\right)\sinh(\beta\well)\sin(\gamma\well)+2\beta\gamma\cosh(\beta\well)\cos(\gamma\well) \right]d=0  \\[6pt]
\left[\beta\left(\beta^2-3\gamma^2+(\nu-2)m^2\right)\sinh(\beta\well)\cos(\gamma\well)
+\gamma\left(\gamma^2-3\beta^2-(\nu-2)m^2\right)\cosh(\beta\well)\sin(\gamma\well)\right]a \\[6pt]
+\!\left[\gamma\left(3\beta^2\!-\!\gamma^2\!+\!(\nu\!-\!2)m^2\right)\sinh(\beta\well)\cos(\gamma\well)
+\beta\left(\beta^2\!-\!3\gamma^2\!+\!(\nu\!-\!2)m^2\right)\cosh(\beta\well)\sin(\gamma\well)\right]d=0 \, ,
\end{cases}
\end{align}
}

{\small
\begin{align} \label{eq:system-2}
\begin{cases}
\left[(\beta^2-\gamma^2-\nu m^2)\sinh(\beta\well)\cos(\gamma\well)-2\beta\gamma\cosh(\beta\well)\sin(\gamma\well) \right]b \\[6pt]
\qquad\qquad +\left[(\beta^2-\gamma^2-\nu m^2)\cosh(\beta\well)\sin(\gamma\well)+2\beta\gamma\sinh(\beta\well)\cos(\gamma\well) \right]c=0 \\[6pt]
\left[\beta\left(\beta^2-3\gamma^2+(\nu-2)m^2\right)\cosh(\beta\well)\cos(\gamma\well)
+\gamma\left(\gamma^2-3\beta^2-(\nu-2)m^2\right)\sinh(\beta\well)\sin(\gamma\well)\right]b \\[6pt]
+\!\left[\gamma\left(3\beta^2\!-\!\gamma^2\!+\!(\nu\!-\!2)m^2\right)\cosh(\beta\well)\cos(\gamma\well)
+\beta\left(\beta^2\!-\!3\gamma^2\!+\!(\nu\!-\!2)m^2\right)\sinh(\beta\well)\sin(\gamma\well)\right]c=0 \, .
\end{cases}
\end{align}
}

We observe that \eqref{eq:system-1} admits a nontrivial solution if and only if the following condition holds true
\begin{align} \label{eq:det-1}
&\gamma\left(\beta^2-\gamma^2-\nu m^2\right)\left[3\beta^2-\gamma^2+(\nu-2)m^2\right]\sinh(\beta\well)\cosh(\beta\well)\\
\notag & \qquad +2\beta\gamma^2\left[3\beta^2-\gamma^2+(\nu-2)m^2\right]\sin(\gamma\well)\cos(\gamma\well) \\
\notag & \qquad -2\beta^2\gamma \left[\beta^2-3\gamma^2+(\nu-2)m^2\right]\sinh(\beta\well)\cosh(\beta\well)  \\
\notag & \qquad +\beta(\beta^2-\gamma^2-\nu m^2)\left[\beta^2-3\gamma^2+(\nu-2)m^2\right]\sin(\gamma\well)\cos(\gamma\well)=0 \, .
\end{align}
We show that \eqref{eq:det-1} is never satisfied.

Exploiting the estimates $\sinh(\beta\well)\cosh(\beta\well)>\beta\well$ and $|\sin(\gamma\well)\cos(\gamma\well)|\le \gamma\well$, \eqref{eq:rel-beta-gamma} and the fact that $\nu\in (0,1)$, we have that
the inequality
\begin{align} \label{eq:ineq-1}
& \gamma\left(\beta^2-\gamma^2-\nu m^2\right)\left[3\beta^2-\gamma^2+(\nu-2)m^2\right]\sinh(\beta\well)\cosh(\beta\well) \\
\notag & \qquad > \left|\beta(\beta^2-\gamma^2-\nu m^2)\left[\beta^2-3\gamma^2+(\nu-2)m^2\right]\sin(\gamma\well)\cos(\gamma\well)\right|
\end{align}
holds if the following one holds true
\begin{align*}
3\beta^2-\gamma^2+(\nu-2)m^2>\left|\beta^2-3\gamma^2+(\nu-2)m^2\right|
\end{align*}
and the validity of this last inequality can be easily verified, thus showing the validity of \eqref{eq:ineq-1}.

On the other hand, we also see that the inequality
{\small
\begin{align} \label{eq:ineq-2}
2\beta^2\gamma \left[-\beta^2+3\gamma^2-(\nu-2)m^2\right]\sinh(\beta\well)\cosh(\beta\well)
>\left|2\beta\gamma^2\left[3\beta^2-\gamma^2+(\nu-2)m^2\right]\sin(\gamma\well)\cos(\gamma\well)\right|
\end{align}
}
holds if the following one holds true
\begin{align*}
\beta^2\left[-\beta^2+3\gamma^2-(\nu-2)m^2\right]>\left|\gamma^2 \left[3\beta^2-\gamma^2+(\nu-2)m^2\right]\right|
\end{align*}
and the validity of this last inequality can be easily verified recalling again \eqref{eq:rel-beta-gamma} and the fact that $\nu\in (0,1)$. This shows the validity of \eqref{eq:ineq-2}.

Combining \eqref{eq:ineq-1} and \eqref{eq:ineq-2} we have that \eqref{eq:det-1} is never satisfied.

We consider now \eqref{eq:system-2} and we observe that it admits a nontrivial solution if and only if the following condition holds true
\begin{align} \label{eq:det-2}
& \gamma\left(\beta^2-\gamma^2-\nu m^2\right)\left[3\beta^2-\gamma^2+(\nu-2)m^2\right]\sinh(\beta\well)\cosh(\beta\well)\\
\notag & \qquad -2\beta\gamma^2\left(3\beta^2-\gamma^2+(\nu-2)m^2\right)\sin(\gamma\well)\cos(\gamma\well) \\
\notag & \qquad -2\beta^2\gamma \left[\beta^2-3\gamma^2+(\nu-2)m^2\right]\sinh(\beta\well)\cosh(\beta\well)  \\
\notag & \qquad -\beta(\beta^2-\gamma^2-\nu m^2)\left[\beta^2-3\gamma^2+(\nu-2)m^2\right]\sin(\gamma\well)\cos(\gamma\well)=0 \, .
\end{align}
With the very same argument adopted for \eqref{eq:det-1}, we can show that condition \eqref{eq:det-2} is never satisfied.

We proved that the case $0<\lambda<\kappa m^4$ does not produce any nontrivial solution to \eqref{eq:ode-y}-\eqref{eq:boundary-hm} and hence no eigenfunctions of \eqref{eq:eigenvalue-normalized}.

\medskip

$\bullet$ {\bf The case $\lambda=\kappa \, m^4$.} Equation \eqref{eq:char} admits the following solutions
\begin{equation*}
\alpha=\pm m
\end{equation*}
and hence the general solution of \eqref{eq:ode-y} is in the form
\begin{equation*}
h_m(y)=a\cosh(my)+b\sinh(my)+cy\cosh(my)+dy\sinh(my) \, .
\end{equation*}
The derivatives of $h_m$ until order three are given by
\begin{align*}
& h_m'(y)=(bm+c)\cosh(my)+(am+d)\sinh(my)+dmy\cosh(my)+cmy\sinh(my) \, , \\[6pt]
& h_m''(y)=m(am+2d)\cosh(my)+m(bm+2c)\sinh(my)+cm^2y\cosh(my)+dm^2y\sinh(my) \, , \\[6pt]
& h_m'''(y)=m^2(bm+3c)\cosh(my)+m^2(am+3d)\sinh(my)+dm^3y\cosh(my)+cm^3y\sinh(my) \, .
\end{align*}
Imposing the boundary conditions \eqref{eq:boundary-hm} we obtain
\begin{equation} \label{eq:system-3}
\begin{cases}
(1-\nu)m^2\cosh(m\well)a+\left[2m\cosh(m\well)+(1-\nu)m^2\well\sinh(m\well)\right]d=0 \\[6pt]
(\nu-1)m^3\sinh(m\well)a+\left[(\nu+1)m^2\sinh(m\well)+(\nu-1)m^3\well \cosh(m\well)\right]d=0 \, ,
\end{cases}
\end{equation}

\begin{equation} \label{eq:system-4}
\begin{cases}
(1-\nu)m^2\sinh(m\well)b+\left[2m\sinh(m\well)+(1-\nu)m^2\well \cosh(m\well)\right]c=0 \\[6pt]
(\nu-1)m^3\cosh(m\well)b+\left[(\nu+1)m^2\cosh(m\well)+(\nu-1)m^3\well \sinh(m\well)\right]c=0 \, .
\end{cases}
\end{equation}
We observe that \eqref{eq:system-3} admits a nontrivial solution if and only if
\begin{align} \label{eq:det-3}
(1-\nu)(\nu+3)m^4\sinh(m\well)\cosh(m\well)-(1-\nu)^2\, m^5 \well=0 \, .
\end{align}
We show that \eqref{eq:det-3} is never satisfied for $\nu\in (0,1)$. Indeed, since $\sinh(m\well)\cosh(m\well)>m\well$, the following inequality
\begin{align} \label{eq:ineq-3}
(1-\nu)(\nu+3)m^4\sinh(m\well)\cosh(m\well)>(1-\nu)^2\, m^5 \well
\end{align}
holds true if the following one holds true
\begin{equation*}
(1-\nu)(\nu+3)m^5 \well>(1-\nu)^2 \, m^5\well \, ;
\end{equation*}
the validity of this inequality is easily verified thus proving the validity of \eqref{eq:ineq-3}.

Now, looking at \eqref{eq:system-4} we have that it admits a nontrivial solution if and only if
\begin{equation} \label{eq:det-4}
(1-\nu)(\nu+3)m^4 \sinh(m\well)\cosh(m\well)+(1-\nu)^2 m^5 \well=0
\end{equation}
but the left hand side of \eqref{eq:det-4} is always positive.

This combined with \eqref{eq:ineq-3} shows that when $\lambda=\kappa m^4$, problem \eqref{eq:ode-y}-\eqref{eq:boundary-hm} does not admit any nontrivial solution.

\medskip

$\bullet$ {\bf The case $\kappa \, m^4<\lambda<(\kappa+1) \, m^4$.} By \eqref{eq:char} we obtain that
\begin{equation} \label{eq:beta-gamma}
\alpha=\pm \beta \ \ \text{or} \ \ \alpha=\pm \gamma \qquad \text{with} \qquad \sqrt{m^2-\sqrt{\lambda-\kappa m^4}}=:\gamma
<\beta:=\sqrt{m^2+\sqrt{\lambda-\kappa m^4}} \, .
\end{equation}
Hence solutions of \eqref{eq:ode-y} are in the form
\begin{equation} 
h_m(y)=a\cosh(\beta y)+b\sinh(\beta y)+c\cosh(\gamma y)+d\sinh(\gamma y)\qquad(a,b,c,d\in\R) \, .
\end{equation}
Imposing \eqref{eq:boundary-hm} we are led to solve the two
systems
\begin{equation} \label{eq:system-5}
\begin{cases}
(\beta^2-m^2\nu)\cosh(\beta\well)a+(\gamma^2-m^2\nu)\cosh(\gamma\well)c=0
\\[7pt]
(\beta^3-m^2(2-\nu)\beta)\sinh(\beta\well)a+(\gamma^3-m^2(2-\nu)\gamma)\sinh(\gamma\well)c=0 \, ,
\end{cases}
\end{equation}

\medskip

\begin{equation} \label{eq:system-6}
\begin{cases}
(\beta^2-m^2\nu)\sinh(\beta\well)b+(\gamma^2-m^2\nu)\sinh(\gamma\well)d=0\\[7pt]
(\beta^3-m^2(2-\nu)\beta)\cosh(\beta\well)b+(\gamma^3-m^2(2-\nu)\gamma)\cosh(\gamma\well)d=0
\, .
\end{cases}
\end{equation}

Proceeding as in \cite[Section 7]{FeGa}, we deduce that
\eqref{eq:system-5} admits a nontrivial solution if and only if
\begin{equation} \label{eq:63}
\tfrac{\beta}{(\beta^2-m^2\nu)^2}\, \tanh(\well\beta) =
\tfrac{\gamma}{(\gamma^2-m^2\nu)^2}\, \tanh(\well\gamma) \, ,
\end{equation}
and \eqref{eq:system-6} admit a nontrivial solution if and only if
\begin{equation} \label{eq:66}
\tfrac{\beta}{(\beta^2-m^2\nu)^2}\,
\coth(\well\beta)
=
\tfrac{\gamma}{(\gamma^2-m^2\nu)^2}\,
\coth(\well\gamma) \, .
\end{equation}
In particular \eqref{eq:ode-y}-\eqref{eq:boundary-hm} admits a
nontrivial solution if and only if at least one of the two
conditions \eqref{eq:63} or \eqref{eq:66} is satisfied.

With the same argument introduced in \cite[Lemma 7.1]{FeGa}, we deduce that there exists a unique
$\lambda=\widetilde\lambda_m^- \in (\kappa \, m^4,(\kappa+1) \, m^4)$ such that \eqref{eq:63} holds. Moreover we also have that
\begin{equation*}
\widetilde\lambda_m^- \in ([(1-\nu)^2+\kappa] \, m^4,(\kappa+1) \, m^4) \, .
\end{equation*}
On the other hand, proceeding as in \cite[Lemma 7.3]{FeGa}, we deduce that there exists a unique $\lambda=\widetilde \lambda_m^+\in (\kappa \, m^4,(\kappa+1)\, m^4)$
such that \eqref{eq:66} holds if and only if
\begin{equation*}
\well m\sqrt 2\, \coth(\well m\sqrt
2)>\left(\tfrac{2-\nu}{\nu}\right)^2 \, .
\end{equation*}
Moreover in such a case we also have
\begin{equation*}
\lambda_m^+ \in ([(1-\nu)^2+\kappa] \, m^4,(\kappa+1) \, m^4) \, .
\end{equation*}

Finally we have that the sequences $m\mapsto \widetilde\lambda_m^-$ and $m\mapsto \widetilde\lambda_m^+$ are increasing as one can show by proceeding as in \cite[Lemmas 7.1-7.5]{FeGa} by putting $\mu=\sqrt \lambda$ and adapting to our case the functions $\Gamma=\Gamma(m,\mu)$, $K=K(m,\mu)$, $\Phi=\Phi(m,\mu)$ and $\Psi=\Psi(m,\mu)$ defined there, exploiting the fact that these four functions can be expressed in terms of $\gamma$; in the present article, the representation of the four functions in terms of $\gamma$ remains unchanged and the only difference is in the expression of $\gamma$ as a function of $m$ and $\mu$, as one can see by \eqref{eq:beta-gamma}.

\medskip

$\bullet$ {\bf The case $\lambda=(\kappa+1) \, m^4$.} By \eqref{eq:char} we deduce that the general solutions of \eqref{eq:ode-y} are in the form
\begin{equation} \label{eq:sol-cy+d}
h_m(y)=a\cosh(\sqrt{2}my)+b\sinh(\sqrt{2}my)+c+dy \, .
\end{equation}
By \eqref{eq:boundary-hm} we are led to the two systems
\begin{equation} \label{eq:system-7}
\begin{cases}
(2-\nu)\cosh(\sqrt{2}m\well)a-\nu c=0 \\[7pt]
\nu\sinh(\sqrt{2}m\well)a=0\, ,
\end{cases}
\end{equation}

\begin{equation} \label{eq:system-8}
\begin{cases}
(2-\nu)\sinh(\sqrt{2}m\well)b-\nu\well d=0 \\[7pt]
\sqrt{2}m\nu\cosh(\sqrt{2}m\well)b+(\nu-2)d=0 \, .
\end{cases}
\end{equation}

Proceeding as in \cite{FeGa}, we deduce that system
\eqref{eq:system-7} admits only the trivial solution $a=c=0$ and
that \eqref{eq:system-8} admits a nontrivial solution if and only
if
\begin{equation} \label{eq:iff}
\tanh(\sqrt{2}m\well)=\left(\frac{\nu}{2-\nu}\right)^2\,
\sqrt{2}m\well\, .
\end{equation}
We recall from \cite{FeGa} that the equation
$\tanh(s)=\left(\frac{\nu}{2-\nu}\right)^2s$ \ admits a
unique solution $\overline{s}>0$. But if
$m_*:=\overline{s}/\well\sqrt2$ is not an integer, then
\eqref{eq:iff} admits no solution. If $m_*\in\N$, then problem
\eqref{eq:ode-y}-\eqref{eq:boundary-hm} admits a nontrivial
solution of the form \eqref{eq:sol-cy+d} with $a=c=0$ and
$(b,d)\neq(0,0)$ whenever $m=m_*$. If $m\in\N$ does not satisfy
\eq{eq:iff}, then problem \eqref{eq:ode-y}-\eqref{eq:boundary-hm}
does not admit any nontrivial solution.

\medskip

$\bullet$ {\bf The case $\lambda>(\kappa+1) \, m^4$.} By
\eqref{eq:char} we deduce that the general solutions of
\eqref{eq:ode-y} are in the form
\begin{equation*}
h_m(y)=a\cosh(\beta y)+b\sinh(\beta y)+c\cos(\gamma
y)+d\sin(\gamma y)
\end{equation*}
where we put
\begin{equation*}
\beta:=\sqrt{\sqrt{\lambda-\kappa m^4}+m^2} \qquad \text{and} \qquad
\gamma:=\sqrt{\sqrt{\lambda-\kappa m^4}-m^2} \, .
\end{equation*}
By \eqref{eq:boundary-hm} we are led to the two systems
\begin{equation} \label{eq:system-9}
\begin{cases}
(\beta^2-m^2\nu)\cosh(\beta\well)a-(\gamma^2+m^2\nu)\cos(\gamma\well)c=0
\\[7pt]
(\beta^3-m^2(2-\nu)\beta)\sinh(\beta\well)a+(\gamma^3+m^2(2-\nu)\gamma)\sin(\gamma\well)c=0\,
,
\end{cases}
\end{equation}

\begin{equation} \label{eq:system-10}
\begin{cases}
(\beta^2-m^2\nu)\sinh(\beta\well)b-(\gamma^2+m^2\nu)\sin(\gamma\well)d=0\\[7pt]
(\beta^3-m^2(2-\nu)\beta)\cosh(\beta\well)b-(\gamma^3+m^2(2-\nu)\gamma)\cos(\gamma\well)d=0
\, .
\end{cases}
\end{equation}
Due to the presence of trigonometric sine and cosine, for any
integer $m$ there exists a sequence $\widetilde \lambda_{j,m} \uparrow +\infty$
such that $\widetilde \lambda_{j,m}>(\kappa+1)m^4$ for all $j\in\N$ and such
that if $\lambda=\widetilde \lambda_{j,m}$ for some $j$ then at least one of the two
systems \eqref{eq:system-9} and \eqref{eq:system-10} admits a
nontrivial solution.




\section{Proof of Theorem \ref{t:conv-poisson}}

We proceed as in the proof of \cite[Theorem 3.2]{FeGa}. Let $w_\ell:=\frac{d^3\mathcal K}{12} \, u_\ell$ and let $\phi$ be the solution of the problem
\begin{equation} \label{eq:phi}
\phi''''=f \quad\mbox{in }(0,L)\,
,\qquad\phi(0)=\phi''(0)=\phi(L)=\phi''(L)=0 \, .
\end{equation}
We expand the function $\phi$ in Fourier series:
\begin{equation} \label{eq:phi-fourier}
\phi(x)=\sum_{m=1}^{+\infty} \frac{L^4\beta_m}{\pi^4 m^4} \,
\sin\left(\frac{m\pi}L \, x\right)  \qquad \text{for
any } x\in (0,L)
\end{equation}
in such a way that we have
\begin{equation}\label{eq:f-Fourier}
  f(x)=\sum_{m=1}^{+\infty} \beta_m \,
\sin\left(\frac{m\pi}L \, x\right)  \qquad \text{for
any } x\in (0,L) \, .
\end{equation}

On the other hand, expanding the function $w_\ell$ we may write
\begin{equation*}
w_\ell(x,y)=\sum_{m=1}^{+\infty} Y_m(y)\sin\left(\frac{m\pi}L \, x\right)  \, .
\end{equation*}
For simplicity we define $m_*=\frac{m\pi}L$ for any $m\ge 1$.
Then we have
\begin{align} \label{eq:Fourier-u}
& f=\Delta^2 w_\ell+\kappa \frac{\partial^4  w_\ell}{\partial
x^4}=\sum_{m=1}^{+\infty} \left[Y''''_m(y)-2m_*^2 \,
Y_m''(y)+(\kappa+1)m_*^4 Y_m(y)\right] \sin\left(m_* x\right) \, .
\end{align}
Comparing \eqref{eq:f-Fourier} and \eqref{eq:Fourier-u} we obtain

\begin{equation} \label{eq:ode-Y}
Y''''_m(y)-2m_*^2 \, Y_m''(y)+(\kappa+1) m_*^4 Y_m(y)=\beta_m \, .
\end{equation}
The characteristic equation is given by
\begin{equation*}
\alpha^4-2m_*^2 \alpha^2+(\kappa+1)m_*^4=0
\end{equation*}
whose solutions are
\begin{equation} \label{eq:beta-gamma-2}
\alpha=\beta\pm i\gamma \quad \text{or} \quad \alpha=-\beta\pm
i\gamma \qquad \text{with} \ \ \beta=m_*
\sqrt{\tfrac{\sqrt{\kappa+1}+1}2} \ \ \text{and} \ \ \gamma=m_*
\sqrt{\tfrac{\sqrt{\kappa+1}-1}2} \, .
\end{equation}
Taking into account that $u_\ell$ is even with respect to $y$, due to the fact that $f$ is independent of $y$, the
even general solution of \eqref{eq:ode-Y} is in the form
\begin{align} \label{eq:sol-Y}
Y_m(y)= a\cosh(\beta y)\cos(\gamma y)+d\sinh(\beta y)\sin(\gamma
y)+\tfrac{\beta_m}{(\kappa+1)m_*^4} \, .
\end{align}
We observe that $\beta^2-\gamma^2=m_*^2$. Moreover, the boundary
conditions in \eqref{eq:model-plate} involving the
horizontal edges imply
\begin{equation} \label{eq:boundary-Ym}
Y_m''(\pm\ell)-\nu m_*^2 \, Y_m(\pm\ell)=0\, , \qquad Y_m'''(\pm\ell)+(\nu-2)m_*^2 \, Y_m'(\pm\ell)=0 \, .
\end{equation}
Combining \eqref{eq:sol-Y} and \eqref{eq:boundary-Ym} and exploiting the computations which brought to \eqref{eq:system-1}, we obtain the system
\begin{align}  \label{eq:system-nonhom}
\begin{cases}
\left[\left(\beta^2-\gamma^2-\nu m_*^2\right)\cosh(\beta\ell)\cos(\gamma\ell)-2\beta\gamma\sinh(\beta\ell)\sin(\gamma\ell)
\right]a \\[6pt]
 \qquad\qquad +\left[\left(\beta^2-\gamma^2-\nu m_*^2\right)\sinh(\beta\ell)\sin(\gamma\ell)+2\beta\gamma\cosh(\beta\ell)\cos(\gamma\ell) \right]d
 =\tfrac{\nu\beta_m}{(\kappa+1)m_*^2}  \\[6pt]
\left[\beta\left(\beta^2-3\gamma^2+(\nu-2)m_*^2\right)\sinh(\beta\ell)\cos(\gamma\ell)
+\gamma\left(\gamma^2-3\beta^2-(\nu-2)m_*^2\right)\cosh(\beta\ell)\sin(\gamma\ell)\right]a \\[6pt]
+\!\left[\gamma\left(3\beta^2\!-\!\gamma^2\!+\!(\nu\!-\!2)m_*^2\right)\sinh(\beta\ell)\cos(\gamma\ell)
+\beta\left(\beta^2\!-\!3\gamma^2\!+\!(\nu\!-\!2)m_*^2\right)\cosh(\beta\ell)\sin(\gamma\ell)\right]d=0 \, .
\end{cases}
\end{align}

With the same argument used for proving that \eqref{eq:det-1} was
never satisfied, but this time with $m_*$ in place of $m$, one can show
that the determinant of the matrix corresponding to system \eqref{eq:system-nonhom} is never zero so that the system always admits a unique solution.

Let us denote by $(a(m,\ell),d(m,\ell))$ the unique solution of
\eqref{eq:system-nonhom}. Passing to the limit as $\ell\to 0$ and replacing the explicit representation of $\beta$ and $\gamma$ according with \eqref{eq:beta-gamma-2}, we obtain

\begin{equation} \label{eq:lim}
\lim_{\ell\to 0}
a(m,\ell)=\frac{\nu^2 \beta_m}{(\kappa+1)(\kappa+1-\nu^2)m_*^4}
\qquad \text{and} \qquad \lim_{\ell\to 0}
d(m,\ell)=\frac{ \nu(\kappa+1-\nu)\beta_m}{\sqrt
\kappa(\kappa+1)(\kappa+1-\nu^2)m_*^4} \, .
\end{equation}
Therefore, proceeding as in the proof of \cite[Theorem 3.3]{FeGa},
by \eqref{eq:sol-Y} and \eqref{eq:lim} and the fact that $-\ell<y<\ell$, we can show that the
``limit function'' of $w_\ell$ as $\ell\to 0$ is
\begin{align} \label{eq:u0}
w_0(x)&:=\sum_{m=1}^{+\infty}
\left(\frac{\nu^2 \beta_m}{(\kappa+1)(\kappa+1-\nu^2)m_*^4}+\frac{\beta_m}{(\kappa+1)m_*^4}\right)\sin(m_*x)
\\[6pt]
\notag & =\sum_{m=1}^{+\infty}
\frac{\beta_m}{(\kappa+1-\nu^2)m_*^4}\sin(m_*x)=\frac{1}{\kappa+1-\nu^2} \, \phi(x)
\end{align}
where the last equality follows by \eqref{eq:phi-fourier}.

When we say that the limit function of $w_\ell$ is $\frac{1}{\kappa+1-\nu^2} \, \phi(x)$ we mean
that
\begin{equation*}
\lim_{\ell \to 0} \, \sup_{(x,y)\in \Omega} \left|w_\ell(x,y)-\frac{1}{\kappa+1-\nu^2} \, \phi(x)\right|=0
\, .
\end{equation*}

Recalling that $(\kappa+1-\nu^2)\mathcal K=E_1$ thanks to \eqref{eq:write-K} and that $E=E_1$ thanks to \eqref{eq:E-KAPPA}, the proof of the theorem now follows combining the definitions of $w_\ell$ and $\phi$ with \eqref{eq:beam-equation-2}.

\section{Proof of Theorem \ref{t:spectral-convergence}}

We follows closely the argument used in \cite[Section 4]{ArFeLa}. For completeness, before the proof of theorem, we report
some basis notions from spectral convergence theory.

Let $\{\mathcal H_\ell\}_{\ell\ge 0}$ be a family of infinite dimensional separable Hilbert spaces and let us assume that there exists a family of linear operators $\mathcal E_\ell:{\mathcal{H}}_0\to {\mathcal{H}}_\ell$ satisfying
\begin{equation} \label{eq:norm-convergence}
\|\mathcal E_\ell \, u\|_{{\mathcal{H}}_\ell} \overset{\eps\to 0}{\longrightarrow}
\|u\|_{{\mathcal{H}}_0} \, , \qquad \text{for all } u\in {\mathcal{H}}_0 \, .
\end{equation}

\begin{definition} \label{d:E-convergence-ArCaLo} Let $\{\mathcal H_\ell\}_{\ell\ge 0}$
and $\{\mathcal E_\ell\}_{\ell\ge 0}$ be as above with $\{\mathcal E_\ell\}_{\ell\ge 0}$ satisfying \eqref{eq:norm-convergence}.

\begin{itemize}
  \item[$(i)$] Let $\{u_\ell\}_{\ell>0}$ be such that $u_\ell\in {\mathcal{H}}_\ell$ for any $\ell>0$. We say that $u_\ell$
$E$-converges to $u\in {\mathcal{H}}_0$ if $\|u_\ell-\mathcal E_\ell \, u\|_{{\mathcal{H}}_\ell}\to 0$
as $\ell\to 0$. We write this as $u_\eps \EC u$.

\medskip

\item[$(ii)$] Let $\{B_\ell\in\mathcal L({\mathcal{H}}_\ell):\ell\ge 0\}$. We say that
$B_\ell$ converges to $B_0$ as $\ell\to 0$ if $B_\ell \, u_\ell \EC B_0 u$ whenever $u_\ell \EC u$. We write
this as $B_\ell \EEC B_0$.

\medskip

\item[$(iii)$] Let $\{B_\ell\in\mathcal L({\mathcal{H}}_\ell):\ell\ge 0\}$ be a family of linear continuous, compact linear operators. We say that $B_\ell$ converges compactly to $B_0$ and we write this as $B_\ell \CC B_0$ if the following two conditions are satisfied:
    \begin{itemize}
      \item[$(a)$] $B_\ell \EEC B_0$ as $\ell\to 0$;

      \medskip

      \item[$(b)$] for any family $\{u_\ell\}_{\ell>0}$ with $u_\ell\in {\mathcal{H}}_\ell$ and
                   $\|u_\ell\|_{\mathcal H_\ell}=1$, there exists a subsequence $B_{\ell_k} \, u_{\ell_k}$ and $w\in \mathcal H_0$ such that $B_{\ell_k} \, u_{\ell_k}\EC w$ as $k\to +\infty$.
    \end{itemize}

\end{itemize}

\end{definition}

Let $\{\mathcal A_\ell\}_{\ell\ge 0}$ be a family of densely defined, closed, symmetric, nonnegative operators with domain $\mathcal D(\mathcal A_\ell)\subset \mathcal H_\ell$. For simplicity we can also assume that $0$ does not belong to the spectrum of $\mathcal A_\eps$ for any $\ell\ge 0$ and that the following two conditions hold true:
\begin{equation} \label{eq:compact-res}
  \mathcal A_\eps \ \text{  has compact resolvent  } \ B_\ell:=\mathcal A_\ell^{-1} \qquad \text{for any } \ell\ge 0 \,
\end{equation}
and
\begin{equation}\label{eq:compact-conv-2}
  B_\ell\CC B_0 \qquad \text{as } \ell\to 0  \quad \text{in the sense of Definition \ref{d:E-convergence-ArCaLo}} \, .
\end{equation}

Under the above conditions on $\{\mathcal A_\ell\}$ we have spectral convergence in the sense of \cite[Theorem 5]{ArFeLa}.
For a simplified version of the spectral convergence result see also \cite[Theorem 1]{FeLa}. More precisely, among the others, we have the following result

\begin{theorem} \label{t:spec-conv-abs}
Let $\{\mathcal A_\ell\}_{\ell\ge 0}$ be a family of densely defined, closed, symmetric, nonnegative operators with domain $\mathcal D(\mathcal A_\ell)\subset \mathcal H_\ell$. Suppose that $0$ does not belong to the spectrum of $\mathcal{A_\ell}$ for any $\ell\ge 0$ and that conditions \eqref{eq:compact-res} and \eqref{eq:compact-conv-2} hold true. Let us denote by
$$
0<\lambda_1^\ell \le \lambda_2^\ell \le \dots \lambda_n^\ell \le \dots
$$
the eigenvalues of $\mathcal A_\ell$ where each eigenvalue is repeated as many times as its multiplicity. Then
$$
\lambda_n^\ell \to \lambda_n^0 \qquad \text{as } \ell \to 0
$$
for any $n\ge 1$.
\end{theorem}

We now introduce some notations with the purpose of applying general spectral theory to our context.

Let $H^2_*(\Omega)$ be the space defined in \eqref{eq:def-H2*} and let $(\cdot,\cdot)_{H^2_*}$ be the scalar product defined
in Lemma \ref{l:equivalence}. We now define the following densely defined, non-negative,
symmetric bilinear form:
\begin{align*}
  & Q_\Omega:\mathcal D(Q_\Omega)\times \mathcal D(Q_\Omega) \subset L^2(\Omega)\times L^2(\Omega)\to \R \, ,  \qquad
    \mathcal D(Q_\Omega)=H^2_*(\Omega)  \, ,\\
  & Q_\Omega(u,v)=\frac{d^3\mathcal K}{12}(u,v)_{H^2_*} \qquad \text{for any } u,v\in \mathcal D(Q_\Omega)\subset L^2(\Omega) \, .
\end{align*}
We also denote by $Q_\Omega$ the corresponding quadratic form defined by
\begin{align*}
    Q_\Omega(u)=Q_\Omega(u,u)  \qquad \text{for any } u\in \mathcal D(Q_\Omega) \, .
\end{align*}
It is a well known fact (see for example the book \cite{Davies}) that since $\mathcal D(Q_\Omega)$ is complete with respect to the norm $Q_\Omega^{1/2}(\cdot)$ then there exists a uniquely determined non-negative selfadjoint operator $H_\Omega$ such that
\begin{equation*}
  \mathcal D(H_\Omega^{1/2})=\mathcal D(Q_\Omega) \quad  \text{and} \quad    Q_\Omega(u,v)=\left(H_\Omega^{1/2}u,H_\Omega^{1/2}v\right)_{L^2} \qquad \text{for any } u,v\in \mathcal D(Q_\Omega)
\end{equation*}
where $H_\Omega^{1/2}$ is the power of the operator $H_\Omega$, which is known to be rigorously defined in the case of non-negative self-adjoint operators like $H_\Omega$.

In such a situation, we have that $u\in \mathcal D(H_\Omega)$ if and only if $u\in \mathcal D(Q_\Omega)$ and there exists $f\in L^2(\Omega)$ such that
\begin{equation*}
  Q_\Omega(u,v)=(f,v)_{L^2} \qquad \text{for any } v\in \mathcal D(Q_\Omega) \, .
\end{equation*}
More explicitly the operator $H_\Omega$ is the differential operator
$\frac{d^3 \mathcal K}{12}\left(\Delta^2+\kappa \frac{\partial^4}{\partial x^4}\right)$ subject to the homogeneous boundary
conditions appearing in \eqref{eq:model-plate}.


  In order to emphasize the dependence of the domain by $\ell$ we write $\Omega_\ell$ in place of $\Omega$. We define the following Hilbert spaces
\begin{equation*}
  \mathcal H_\ell=L^2(\Omega_\ell\, ;(2\ell)^{-1} dxdy) \qquad \text{and} \qquad \mathcal H_0=L^2(0,L) \, .
\end{equation*}

For any $\ell>0$, let $H_{\Omega_\ell}$ be the operator defined above with domain $\mathcal D(H_{\Omega_\ell})\subset \mathcal H_\ell$ and let us put $\mathcal A_\ell=H_{\Omega_\ell}$.
Moreover, for $\ell=0$ we define the operator $\mathcal A_0:\mathcal D(\mathcal A_0)\subset L^2(0,L)\to L^2(0,L)$ as the one
associated with the bilinear form
\begin{align*}
  & Q_0:\mathcal D(Q_0)\times \mathcal D(Q_0)\subset L^2(0,L)\times L^2(0,L)\to \R \, ,
  \qquad \mathcal D(Q_0)=H^2(0,L)\cap H^1_0(0,L) \\
  & Q_0(u,v)=\frac{d^3\mathcal K(\kappa+1-\nu^2)}{12} \int_{0}^{L} u''(x) v''(x)\,  dx \qquad \text{for any } u,v\in \mathcal D(Q_0)\subset L^2(0,L) \, .
\end{align*}
In other words, once we put $E=(\kappa+1-\nu^2)\mathcal{K}$, $A_0$ becomes the differential operator $\frac{Ed^3}{12} \frac{d^4}{dx^4}$ with homogeneous Navier boundary conditions.

We shall show that $\{\mathcal A_\ell\}_{\ell\ge 0}$ satisfies all the assumptions of Theorem \ref{t:spec-conv-abs}, thus proving in this way the validity of Theorem \ref{t:spectral-convergence}.

The fact that $\mathcal A_\ell$ satisfies \eqref{eq:compact-res} for any $\ell\ge 0$, is an easy consequence of Lemma \ref{l:equivalence}, Theorem \ref{t:Lax-Milgram} and the compact embedding $H^2(\Omega_\ell)\subset L^2(\Omega_\ell)$.

We now construct a family of operators $\{\mathcal E_\ell\}_{\ell>0}$ satisfying \eqref{eq:norm-convergence}.
The operators $\mathcal E_\ell:\mathcal{H}_0\to \mathcal{H}_\ell$ are simply defined by
\begin{equation} 
  (\mathcal E_\ell \, v)(x,y)=v(x) \, , \qquad (x,y)\in (0,L)\times (-\ell,\ell) \, ,
\end{equation}
for any $v\in \mathcal H_0=L^2(0,L)$.
Dy direct calculation one sees that $\|\mathcal E_\ell \,  v\|_{\mathcal H_\ell}=\|v\|_{\mathcal H_0}$ thus showing
that \eqref{eq:norm-convergence} is trivially satisfied.

The next step is to prove the validity of \eqref{eq:compact-conv-2}. Once \eqref{eq:compact-conv-2} is proved, the proof of Theorem \ref{t:spectral-convergence} follows from Theorem \ref{t:spec-conv-abs}.

The proof of \eqref{eq:compact-conv-2} is contained in the following

\begin{lemma} \label{l:comp-conv}
For any $\ell\ge 0$ let $\mathcal A_\ell=H_{\Omega_\ell}$ be as above and let us denote by $B_\ell=\mathcal A_\ell^{-1}$ the corresponding resolvent operators. Then $B_\ell\CC B_0$ as $\ell\to 0$ in the sense of Definition \ref{d:E-convergence-ArCaLo}.
\end{lemma}

\begin{proof}
  The proof of this lemma follows closely the arguments contained in \cite[Section 4] {ArFeLa}.
  Let $\{f_\ell\}_{\ell>0}$ be such that $\|f_\ell\|_{\mathcal H_\ell}=1$. Let $\Omega_1=(0,L)\times (-1,1)$ and let
  $\tilde f_\ell\in L^2(\Omega_1)$ be defined by $\tilde f_\ell(x,y):=f_\ell(x,\ell y)$ for any $(x,y)\in \Omega_1$.
  Then $\|\tilde f_\ell\|_{L^2(\Omega_1)}=\sqrt 2 \, \|f_\ell\|_{\mathcal H_\ell}=\sqrt 2$ for any $\ell>0$ which shows that
  $\{\tilde f_\ell\}$ is bounded in $L^2(\Omega_1)$.

  Let us denote by $u_\ell$ the unique solution of \eqref{eq:model-plate} in $\Omega_\ell$ with forcing term $f_\ell$ and let us
  define $\tilde u_\ell(x,y):=u_\ell(x,\ell y)$ for any $(x,y)\in \Omega_1$ in such a way that
  \begin{equation*}
    \begin{cases}
       \frac{d^3\mathcal K}{12} \left((1+\kappa)\frac{\partial^4 \tilde u_\ell}{\partial x^4}
                                      +\frac{2}{\ell^2} \frac{\partial^4 \tilde u_\ell}{\partial x^2 \partial y^2}
                                      +\frac{1}{\ell^4} \frac{\partial^4 \tilde u_\ell}{\partial y^4}
        \right)=\tilde f_\ell
          & \qquad \text{in } \Omega_1 \, , \\[6pt]
       \tilde u_\ell(0,y)=\frac{\partial^2 \tilde u_{\ell}}{\partial x^2}(0,y)
         =\tilde u_\ell(L,y)=\frac{\partial^2 \tilde u_{\ell}}{\partial x^2}(L,y)=0
         & \qquad \text{for } y\in(-1,1) \, , \\[6pt]
       \frac{\partial^2 \tilde u_{\ell}}{\partial x^2}(x,\pm 1)
         +\frac{\nu}{\ell^2} \frac{\partial^2 \tilde u_{\ell}}{\partial y^2}(x,\pm 1)=0 & \qquad \text{for }
          x\in (0,L) \, , \\[6pt]
          \frac{1}{\ell^3} \frac{\partial^3 \tilde u_{\ell}}{\partial y^3}(x,\pm 1)
            +\frac{(2-\nu)}\ell \frac{\partial^3 \tilde u_{\ell}}{\partial x^2 \partial y}(x,\pm 1)=0 & \qquad
         \text{for } x\in (0,L) \, .
\end{cases}
  \end{equation*}

The corresponding variational formulation is
{\small
\begin{align} \label{eq:var-form-scaled}
   & \tfrac{d^3\mathcal K}{12} \int_{\Omega_1} \left[\nu \left( \tfrac{\partial^2 \tilde u_{\ell}}{\partial x^2}
      +\tfrac{1}{\ell^2} \tfrac{\partial^2 \tilde u_{\ell}}{\partial y^2} \right)
      \left( \tfrac{\partial^2 \varphi}{\partial x^2}
      +\tfrac{1}{\ell^2} \tfrac{\partial^2 \varphi}{\partial y^2} \right)
       \right] dxdy \\[6pt]
   & \notag \qquad
     +\tfrac{d^3\mathcal K}{12}(1-\nu) \int_{\Omega_1} \left[\tfrac{\partial^2 \tilde u_{\ell}}{\partial x^2}
     \tfrac{\partial^2 \varphi}{\partial x^2}
     +\tfrac{2}{\ell^2}\tfrac{\partial^2 \tilde u_{\ell}}{\partial x \partial y}\tfrac{\partial^2 \varphi}{\partial x \partial y}
     +\tfrac{1}{\ell^4}\tfrac{\partial^2 \tilde u_{\ell}}{\partial y^2}\tfrac{\partial^2 \varphi}{\partial y^2}    \right]dxdy
      \\[6pt]
   & \notag \qquad
      +\tfrac{d^3\mathcal K}{12} \kappa \int_{\Omega_1} \tfrac{\partial^2 \tilde u_{\ell}}{\partial x^2}
     \tfrac{\partial^2 \varphi}{\partial x^2} \, dxdy = \int_{\Omega_1} \tilde f_\ell \, \varphi \, dxdy
     \qquad \text{for any } \varphi \in H^2_*(\Omega_1) \, .
\end{align}
}
Testing \eqref{eq:var-form-scaled} with $\varphi=\tilde u_\ell$, exploiting for $0<\ell<1$ the estimate
{\small
\begin{align*} 
   \left|\int_{\Omega_1} \tilde f_\ell \, \tilde u_\ell \, dxdy\right| & \le
   \|\tilde f_\ell\|_{L^2(\Omega_1)} \|\tilde u_\ell\|_{L^2(\Omega_1)}  \\[7pt]
  & \le C\|\tilde f_\ell\|_{L^2(\Omega_1)} \left(\int_{\Omega_1} |D^2 \tilde u_\ell|^2 dxdy\right)^{\frac 12}
    \quad  \text{(see the proof of \cite[Lemma 4.1]{FeGa})}\\[7pt]
   & \le C\|\tilde f_\ell\|_{L^2(\Omega_1)} \left\{\int_{\Omega_1}
      \left[\left(\tfrac{\partial^2 \tilde u_{\ell}}{\partial x^2}\right)^2
     +\tfrac{2}{\ell^2}\left(\tfrac{\partial^2 \tilde u_{\ell}}{\partial x \partial y}\right)^2
     +\tfrac{1}{\ell^4}\left(\tfrac{\partial^2 \tilde u_{\ell}}{\partial y^2}\right)^2\right]
     dxdy \right\}^{\frac 12}
\end{align*}
}
and using the Young inequality one deduce that
\begin{equation} \label{eq:van-D2}
  \left\|\frac{\partial^2 \tilde u_\ell}{\partial x^2}\right\|_{L^2(\Omega_1)}=O(1) \, , \quad
  \left\|\frac{\partial^2 \tilde u_\ell}{\partial x \partial y}\right\|_{L^2(\Omega_1)}=O(\ell) \, , \quad
  \left\|\frac{\partial^2 \tilde u_\ell}{\partial y^2}\right\|_{L^2(\Omega_1)}=O(\ell^2) \, ,
   \quad \text{as } \ell \to 0 \, .
\end{equation}
In particular we have that $\{\tilde u_\ell\}_{0<\ell<1}$ is bounded in $H^2(\Omega_1)$ and hence there exists
$u\in H^2(\Omega_1)$ such that $\tilde u_\ell \rightharpoonup u$ weakly in $H^2(\Omega_1)$ as $\ell \to 0$ along a sequence. Actually $u\in H^2_*(\Omega_1)$ being $H^2_*(\Omega_1)$ a closed subspace of $H^2(\Omega_1)$.

Hence by \eqref{eq:van-D2}, for any test function $\varphi \in C^\infty_c(\Omega_1)$ we have that
\begin{equation*}
  \int_{\Omega_1} u \frac{\partial^2 \varphi}{\partial x \partial y} \, dxdy
    =\lim_{\ell \to 0}   \int_{\Omega_1} \tilde u_\ell \frac{\partial^2 \varphi}{\partial x \partial y} \, dxdy
    =\lim_{\ell \to 0}   \int_{\Omega_1} \frac{\partial^2 \tilde u_\ell}{\partial x \partial y}\, \varphi \, dxdy=0
\end{equation*}
which implies that $\frac{\partial^2 u}{\partial x\partial y}\equiv 0$ in $\Omega_1$. In particular we have that there exist
two functions $\psi_1\in H^2(0,L)$ and $\psi_2\in H^2(-1,1)$ such that $u(x,y)=\psi_1(x)+\psi_2(y)$ for any $(x,y)\in \Omega_1$.

But $u\in H^2_*(\Omega_1)$ so that $0=u(0,y)=\psi_1(0)+\psi_2(y)$ for any $y\in (-1,1)$ which shows that $\psi_2$ is constant.
Hence the function $u$ does not depend on $y$.

On the other hand, by \eqref{eq:van-D2} we also deduce that there exists $w\in L^2(\Omega_1)$ such that
$\frac{1}{\ell^2} \frac{\partial^2 \tilde u_\ell}{\partial y^2}\rightharpoonup w$ weakly in $L^2(\Omega_1)$ as $\ell\to 0$ along a sequence.

Resuming what we stated above, we can find a common sequence $\ell_k\to 0$ such that the three conditions hold true simultaneously
\begin{equation*}
  \tilde u_{\ell_k} \rightharpoonup u \quad \text{weakly in } H^2(\Omega_1) \, ,
  \qquad \frac{1}{\ell_k^2} \frac{\partial^2 \tilde u_{\ell_k}}{\partial y^2}\rightharpoonup w
  \quad \text{weakly in } L^2(\Omega_1) \, , \qquad \tilde f_{\ell_k}\rightharpoonup f \quad \text{weakly in } L^2(\Omega_1)
\end{equation*}
for some function $f\in L^2(\Omega_1)$.
Using test functions in \eqref{eq:var-form-scaled} only depending on $x$, we infer that $u$ solves
\begin{equation*}
 \frac{d^3 \mathcal K}{12} \int_0^L \left[(1+\kappa)u''(x) \varphi''(x)+\nu \mathcal M(w)\varphi''(x) \right]\, dx
  =\int_{0}^{L} \mathcal M(f) \varphi(x)\, dx
\end{equation*}
for any $\varphi\in H^2(0,L)\cap H^1_0(0,L)$, where we put $(\mathcal M h)(x)=\frac{1}{2} \int_{-1}^{1} h(x,y)\, dy$ for any function $h\in L^2(\Omega_1)$.

With the same argument used in \cite[Section 4.1]{ArFeLa}, one can prove that $w$ does not depend on $y$ and moreover
$w=-\nu \frac{\partial^2 u}{\partial x^2}$ thus proving that
\begin{equation*}
 \frac{d^3 \mathcal K(1+\kappa-\nu^2)}{12} \int_0^L u''(x) \varphi''(x)\, dx
  =\int_{0}^{L} \mathcal M(f) \varphi(x)\, dx \qquad \text{for any } \varphi\in H^2(0,L)\cap H^1_0(0,L) \, .
\end{equation*}
Recalling \eqref{eq:write-K} and \eqref{eq:E-KAPPA}, we infer that $u$ solves
\begin{equation} \label{eq:boo}
  \begin{cases}
  \frac{Ed^3}{12} \, u''''=\mathcal M(f) \qquad \text{in } (0,L) \, , \\[6pt]
  u(0)=u(L)=u''(0)=u''(L) \, .
  \end{cases}
\end{equation}
By compact embedding $H^2(\Omega_1)\subset L^2(\Omega_1)$ we have that $\tilde u_{\ell_k}\to u$ strongly in $L^2(\Omega_1)$. Recalling the definition of $\tilde u_\ell$ and of the operators $\mathcal E_\ell$, the fact that $u$ may be seen as a function in $L^2(0,L)$, after a change of variable we thus obtain
\begin{equation*}
  \|u_{\ell_k}-\mathcal E_{\ell_k} u\|_{\mathcal H_{\ell_k}}^2 \!=\! \frac{1}{2\ell_k} \int_{\Omega_{\ell_k}} \!\! |u_{\ell_k}(x,y)-u(x)|^2 dxdy \! = \! \frac 12 \int_{\Omega_1} \!\! |\tilde u_{\ell_k}(x,y)-u(x)|^2 dxdy \to 0
  \quad \text{as } k\to +\infty \, .
\end{equation*}
In other words, we have just proved that if $\{f_\ell\}_{\ell>0}$ is such that $\|f_\ell\|_{\mathcal H_\ell}=1$ then there exists a sequence $\ell_k$ converging to zero and a function $u\in \mathcal H_0$ such that $B_{\ell_k} f_{\ell_k}=u_{\ell_k} \EC u$.

In order to complete the proof of the compact convergence we have to prove that $B_\ell \EEC B_0$ as $\ell\to 0$.

To this purpose, suppose that $f_\ell \EC f$ as $\ell\to 0$ so that
\begin{equation*}
  o(1)=\frac{1}{2\ell} \int_{\Omega_\ell} |f_\ell(x,y)-f(x)|^2 dxdy=\frac 12 \|\tilde f_\ell-f\|_{L^2(\Omega_1)}^2 \, .
\end{equation*}
Proceeding as above we find a function $u\in \mathcal H_0$ such that $B_\ell f_\ell \EC u$ along a sequence and moreover $u$ solves \eqref{eq:boo} with $\mathcal M f=f$ being $f$ independent of $y$. By uniqueness of solutions of \eqref{eq:boo} we have that $u=B_0 f$ and hence $B_\ell f_\ell \EC B_0 f$ not only along a sequence but as $\ell\to 0$ in the usual sense. This completes the proof of the convergence $B_\ell \EEC B_0$ and hence of the lemma.
\end{proof}



{\bf Acknowledgments} The author is member of the Gruppo Nazionale per l'Analisi Matematica, la Probabilit\`{a} e le loro Applicazioni (GNAMPA) of the Istituto Nazionale di Alta Matematica (INdAM). The author acknowledges partial financial support from the PRIN project 2017 ``Direct and inverse problems for partial differential equations: theoretical aspects and
applications'' and from the INDAM - GNAMPA project 2019 ``Analisi spettrale per operatori ellittici del secondo e quarto ordine con condizioni al contorno di tipo Steklov o di tipo parzialmente incernierato''.

This research was partially supported by the research project ``Metodi e modelli per la matematica e le sue
applicazioni alle scienze, alla tecnologia e alla formazione'' Progetto di Ateneo 2019 of the University of Piemonte Orientale ``Amedeo Avogadro''. 

The author is grateful to Elvise Berchio, Alessio Falocchi and Pier Domenico Lamberti for the useful discussions and suggestions that supported the beginning of this work.

\end{document}